\documentclass[12pt]{amsart}

\usepackage[pdfauthor   = {Lukas\ Kuehne\ and\ Geva\ Yashfe},
            pdftitle    = {Representability\ of\ matroids\ by\ c-arrangements\ is\ undecidable},
            pdfsubject  = {},
            pdfkeywords = {matroids;\ multilinear matroids;\ $c$-arrangements;\ undecidability;\ Dowling geometry},
            bookmarks=true,
            bookmarksopen=true,
            pagebackref=true,
            hyperindex=true,
            colorlinks=true,
            linkcolor=blue,
            citecolor=blue,
            filecolor=blue,
            urlcolor=blue,
            ]{hyperref}

\usepackage[utf8]{inputenc}
\usepackage[T1]{fontenc}
\usepackage[a4paper, includeheadfoot, margin=2.81cm, bottom=1cm]{geometry}                
\usepackage{times}
\usepackage{mathrsfs}
\usepackage{mathtools}
\usepackage{latexsym}
\usepackage{amssymb}
\usepackage{amsthm}
\usepackage{epsfig}
\usepackage{colortbl}
\usepackage{fancyvrb}
\usepackage{graphicx}
\usepackage{subcaption}
\usepackage[font=small,labelfont=bf]{caption}
\usepackage[dvipsnames]{xcolor}
\usepackage{accents} 
\usepackage{enumerate}
\usepackage{blkarray, bigstrut}
\usepackage{xparse}

\usepackage{wrapfig}
\usepackage{tikz}
\usetikzlibrary{shapes,arrows,matrix,backgrounds,positioning,plotmarks,calc,patterns,
decorations.shapes,
decorations.fractals,
decorations.markings,
decorations.pathreplacing,
decorations.pathmorphing,
decorations.text
}
\usepackage[colorinlistoftodos,shadow]{todonotes}
\usepackage{bm}
\usepackage{multirow}
\usepackage{mdwlist}
\usepackage{stmaryrd}
\usepackage{mathdots} 

\usepackage[toc,page]{appendix}
\usepackage{float}

\usepackage{bigfoot}

\usepackage[sort&compress,capitalise]{cleveref}

\crefname{exmp}{Example}{Examples}

\usepackage[linesnumbered,commentsnumbered,ruled,vlined]{algorithm2e}

\SetCommentSty{mycommfont}

\newtheoremstyle{mytheoremstyle} 
    {5pt}                    
    {5pt}                    
    {\itshape}                   
    {\parindent}                           
    {\bf}                   
    {.}                          
    {.5em}                       
    {}  

\theoremstyle{mytheoremstyle}

\newtheorem{theorem}{Theorem}[section]

\newtheorem{lemm}[theorem]{Lemma}
\newtheorem{prop}[theorem]{Proposition}
\newtheorem{coro}[theorem]{Corollary}
\newtheorem{prob}[theorem]{Problem}

\newcounter{claims}
\newtheorem{claim}[claims]{Claim}
\makeatletter
\@addtoreset{claims}{section}
\makeatother

\newtheoremstyle{mytdefintionstyle} 
    {5pt}                    
    {5pt}                    
    {\rm}                   
    {\parindent}                           
    {\bf}                   
    {.}                          
    {.5em}                       
    {}  

\theoremstyle{remark}
\newtheorem{rmrk}[theorem]{Remark}

\theoremstyle{mytdefintionstyle}
\newtheorem{defn}[theorem]{Definition}

\newtheoremstyle{exmp_contd}
    {5pt}                    
    {5pt}                    
    {\rm}                   
    {\parindent}                           
    {\bf}                   
    {.}                          
    {.5em}                       
    {\thmname{#1}\ \thmnumber{ #2}\thmnote{#3}\ (continued)}  
\theoremstyle{exmp_contd}

\DeclareMathOperator{\de}{def}

\DeclareMathOperator{\spa}{span}

\newcommand\A{\mathcal{A}}
\newcommand\C{\mathbb{C}}

\newcommand\F{\mathbb{F}}

\newcommand{\Q}{\mathbb{Q}}

\newcommand\R{\mathbb{R}}

\newcommand\N{\mathbb{N}}
\renewcommand\phi{\varphi}
\DeclareMathOperator\rk{rk}

\newcommand{\U}{\mathcal{U}}

\definecolor{darkgray}{rgb}{0.3,0.3,0.3}
\definecolor{LightGray}{gray}{0.9}

\newcommand{\topstrut}[1][1.2ex]{\setlength\bigstrutjot{#1}{\bigstrut[t]}}
\newcommand{\botstrut}[1][0.9ex]{\setlength\bigstrutjot{#1}{\bigstrut[b]}}

\DeclareDocumentEnvironment{bheadmatrix}{O{}O{*{10}{c}}}{%
	\begin{blockarray}{#2}
		#1 \\[-0.8ex]
		\begin{block}{[#2]}
			\topstrut}%
		{%
			\botstrut\\
		\end{block}
	\end{blockarray}
}%

\definecolor{darkgreen}{rgb}{0.008,0.617,0.067}
\definecolor{brown}{rgb}{0.6,0.4,0.2}

\newif\ifjournalversion

\author{Lukas K\"uhne}
\author{Geva Yashfe}
\address{Einstein Institute of Mathematics, The Hebrew University of Jerusalem, Giv’at Ram, Jerusalem, 91904, Israel}
\email{\href{mailto:Lukas Kuehne<lukas.kuehne@mail.huji.ac.il>}{lukas.kuehne@mail.huji.ac.il}, \href{mailto:Geva Yashfe <geva.yashfe@mail.huji.ac.il>}{geva.yashfe@mail.huji.ac.il}}

\begin{document}

\title{Representability of matroids by $\bm{c}$-arrangements is undecidable}
\begin{abstract}
	For a natural number $c$, a $c$-arrangement is an arrangement of dimension $c$ subspaces satisfying the following condition: the sum of any subset of the subspaces has dimension a multiple of $c$.
	Matroids arising as normalized rank functions of $c$-arrangements are also known as multilinear matroids.
	We prove that it is algorithmically undecidable whether there exists a $c$ such that a given matroid has a $c$-arrangement representation, or equivalently whether the matroid is multilinear.
	It follows that certain network coding problems are also undecidable.
	In the proof, we introduce a generalized Dowling geometry to encode an instance of the uniform word problem for finite groups in matroids of rank three.
	The $c$-arrangement condition gives rise to some difficulties and their resolution is the main part of the paper.
	
\end{abstract}

\thanks{L.K. was supported by a Minerva fellowship of the Max-Planck-Society, the Studienstiftung des deutschen Volkes and by ERC StG 716424 - CASe. G.Y. was supported by ISF grant 1050/16.}

\maketitle

\vspace{-1cm}
\section{Introduction} \label{sec:Intro}
\subsection{$c$-Arrangement Representations}

The main objects discussed in this article are matroids and their generalizations polymatroids.
\begin{defn}\label{def:matroid}
A \emph{polymatroid} is a pair of a \emph{ground set} $E$ together with a \emph{rank function} $r:\mathcal{P}(E) \rightarrow \R_{\ge 0}$ which is
\begin{enumerate}[(i)]
	\item \emph{monotone}: $r(S) \leq r(T)$ for each $S\subseteq T \subseteq E$ and 
	\item \emph{submodular}: $r(S)+r(T) \ge r(S\cup T)+ r(S\cap T)$ for each $S,T\subseteq E$.
\end{enumerate}
The pair $(E,r)$ is called a \emph{matroid} if $r$ only takes integer values and additionally $r(S)\le |S|$ holds for any subset $S\subseteq E$.
\end{defn}
It is a classical problem to study matroid representations by vector configurations or equivalently hyperplane arrangements over some field, for an overview cf.\ \cite[Chapter~6]{Oxl11}.
Goresky and MacPherson extended this notion by introducing \emph{$c$-arrangements} in the context of stratified Morse theory~\cite{GM88}.
For a fixed integer $c\ge 1$, these are arrangements of dimension $c$ subspaces of a vector space such that the dimension of each sum of these subspaces has dimension a multiple of $c$.
This condition ensures that the associated rank function normalized by $\frac{1}{c }$ is the rank function of a matroid.
A matroid arising in this way is said to be \emph{representable as a $c$-arrangement}.
This generalizes the usual notion of matroid representations which are just $1$-arrangements.
Goresky and MacPherson showed for instance that the non-Pappus matroid which is not representable over any field is representable as a 2-arrangement over $\C$.
Matroids arising as representations of $c$-arrangement representations are also called \emph{multilinear matroids}.

Fix a field $\F$. The following is the multilinear representability problem over $\F$:
\begin{prob}\label{prob:c_representability}
	Given a matroid $M$, does there exist a $c\ge 1$ such that $M$ is representable as a $c$-arrangement over $\F$?
\end{prob}
This question was posed by Bj\"orner where he states ``the question of $c$-representability of matroids is open, but probably hopeless.''~\cite[p. 333]{Bjo92}.
The main contribution of this article is a computability theoretic result for $c$-arrangement representations.
\begin{theorem}\label{thm:undecidable_matroids}
The multilinear representability problem is undecidable.
This is true for any field $\F$. 
Moreover, the problem remains undecidable if the field remains unspecified, or is allowed to be taken from some given set.
\end{theorem}
In this sense, Bj\"orner was correct; the representability question is indeed ``hopeless''.

\subsection{Related Work}
Multilinear matroids found applications to network coding capacity:
In~\cite{Ahl00}, Ahlswede et al.\ introduced a model for network information flow problems.
When the coding functions are constrained to be linear, these problems are related to polymatroid representability.
In~\cite{ElR10}, El Rouayheb et al.\ constructed linear network capacity problems equivalent to multilinear matroid representability (cf.\ also~\cite{DFZ07} for a related construction).
\Cref{thm:undecidable_matroids} together with Proposition 18 in~\cite{ElR10} thus implies that the question whether an instance of the network coding problem has a  linear vector coding solution is also undecidable.
Another application is to secret sharing schemes, see~\cite{SA98}.

One can determine whether a given matroid is representable as a $1$-arrangement over an algebraically closed field via a Gr\"obner basis computation (cf.~\cite[p. 227]{Oxl11}). 
In fact, the same method can be adapted to decide $c$-arrangement representability for a fixed $c\ge 1$.
In his Ph.D.\ thesis, Mn\"ev proved a universality theorem for realization spaces of oriented matroids, cf.\ \cite{Mne88} for an exposition.
Subsequently, Sturmfels observed that $1$-arrangement representability over the rational numbers $\Q$ is equivalent to Hilbert's Tenth Problem for $\Q$, which asks whether a single multivariate polynomial equation over the integers has a solution in $\Q$~\cite{Stu87}.
It is not known whether this is decidable.

An important class of matroids are \emph{Dowling geometries} which are defined from finite groups~\cite{Dow73}.
In~\cite{BBP14}, Beimel et al.\ characterized when a Dowling geometry is representable as a $c$-arrangement in terms of fixed point free representations of its underlying group.
Our work is related to Dowling geometries and the above characterization: we generalize Dowling's construction and construct matroids which encode finitely-presented groups.
The issues we then have to deal with are directly related to the existence of fixed points in matrix representations of these groups.

Multilinear matroid representations are special cases of two more general classes of matroid representations.
One of these is the class of matroids representable over a \emph{skew partial field}, where a $c$-arrangement representation over a field $\F$ is equivalent to a skew partial field representation over the matrix ring $M_c(\F)$~\cite{PvZ13}.
The other is the class of \emph{entropic matroids}~\cite{Fuj78} and their equivalent definitions via \emph{partition representations}~\cite{Mat93} or \emph{secret sharing matroids}~\cite{BD91}.
These also contain the class of multilinear matroids.
In both cases it is conjectured that the inclusion of multilinear matroids in these classes is strict.
We will present an example of a matroid that is skew partial field representable but not multilinear in forthcoming joint work with Rudi Pendavingh~\cite{KPY20}.

\subsection{Uniform Word Problem for Finite Groups}

The proof of \Cref{thm:undecidable_matroids} relates the $c$-arrangement representability to the \emph{uniform word problem for finite groups} (UWPFG):
\begin{description}
	\item[Instance] A finite presentation $\langle S\mid R\rangle$ of a group, that is $S$ is a finite set of generators and $R$ a finite set of relations in $S$, together with a word $w$ which is a product of elements in $S$ and their inverses.
	\item[Question] Does every group homomorphism from the group defined by the presentation $\langle S\mid R\rangle $ to a \emph{finite} group $G$ map $w$ to the identity in $G$?
\end{description}

\begin{theorem}[\cite{Slo81}]\label{thm:undecidable_grp_theory}
	The uniform word problem for finite groups is undecidable.
\end{theorem}

\subsection{Structure of the Paper} This paper is organized as follows.
In \Cref{sec:outline} we give a high-level outline of the proof.
\Cref{sec:pre} gives definitions and basic properties which are used throughout the article.
Generalized Dowling geometries and their relation to the uniform word problem are explained in~\Cref{sec:staudt}.
In \Cref{sec:alg_reg} to \Cref{sec:c-bases} we develop several technical tools as explained in the outline.
We put these tools together to prove \Cref{thm:undecidable_matroids} in \Cref{sec:main_proof}.

\section*{Acknowledgements}
We would like to thank Karim Adiprasito for his mentorship and for introducing us to the topic of $c$-arrangements.
Furthermore, we are grateful to Rudi Pendavingh for helpful conversations on von Staudt constructions. Our application of them is inspired by joint work with him in~\cite{KPY20}.

An extended abstract of this paper will appear as ``Undecidability of $c$-Arrangement Matroid
Representations" in a proceedings volume of Séminaire Lotharingien de Combinatoire, among extended abstracts from the 2020 FPSAC conference.

\section{Outline of the Proof}\label{sec:outline}

The first idea in the proof is a generalized Dowling geometry, which is a matroid that encodes a set of multiplicative relations between matrices of unspecified size (had we been in the setting of linear representations of matroids, we would have had scalars from a field instead). This is essentially a von Staudt construction which encodes only multiplicative relations (the Dowling geometries, originally introduced in~\cite{Dow73}, are a special case).

\begin{figure}[bh!]
	\centering
	\includegraphics[width=0.5\linewidth]{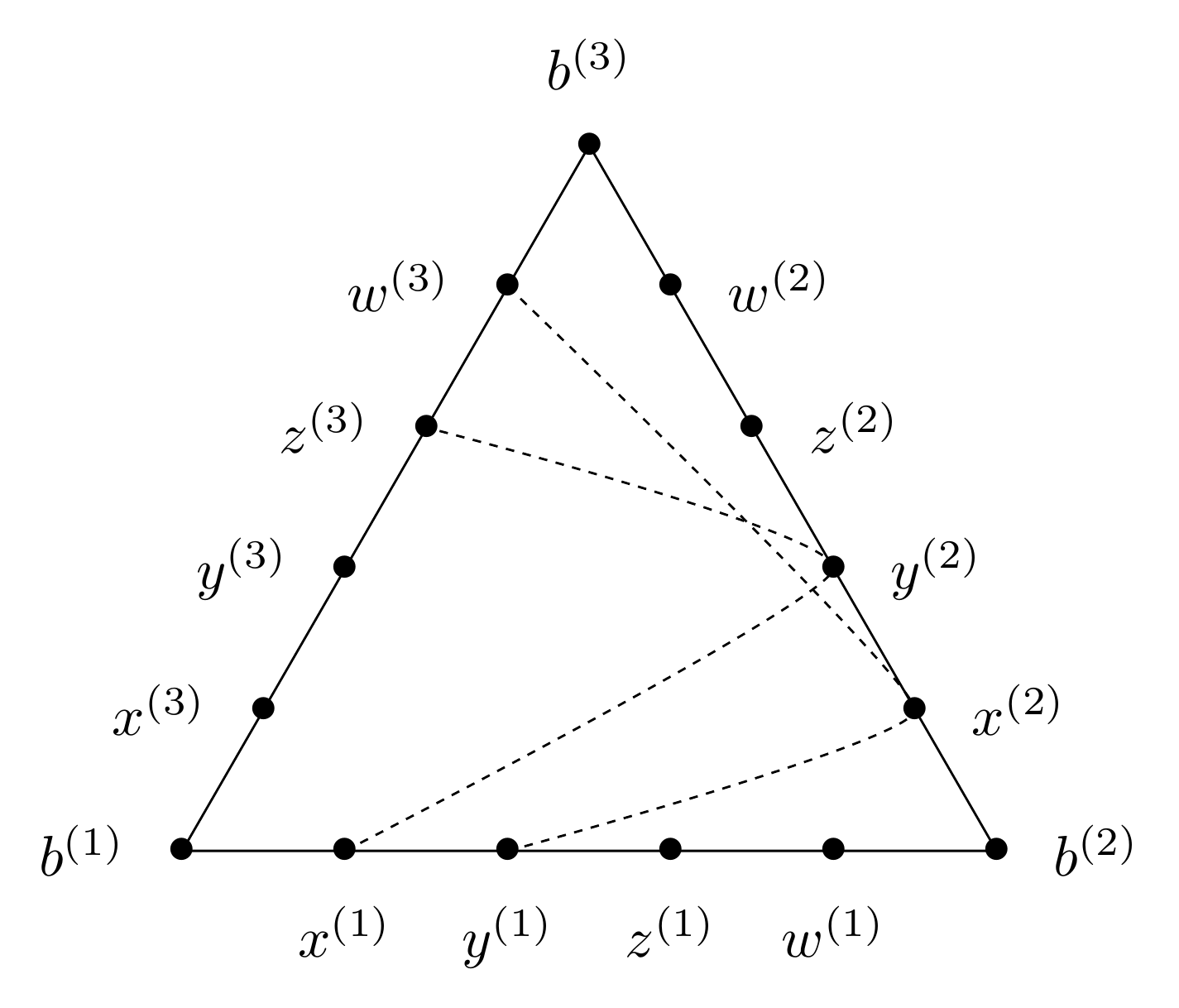}
	\caption{A geometric depiction of a triangle matroid. Each side contains copies of the generators $x,y,y,z$. The triples $\{x^{(1)},y^{(2)},z^{(3)}\}$ and $\{y^{(1)},x^{(2)},w^{(3)}\}$ form two circuits of the matroid.}
	\label{fig:triangle_matroid}
\end{figure}

The construction of generalized Dowling geometries is carried out in \cref{sec:staudt}. It takes as input an instance $\langle S\mid R\rangle, w$ of the UWPFG (uniform word problem for finite groups) and outputs a matroid $N_{S,R}$.

The matroid $N_{S,R}$ is a triangle matroid, or frame matroid of rank three, which means it is is a matroid of rank three with a distinguished basis such that all other elements are on the lines spanned by this basis.
\Cref{fig:triangle_matroid} shows a geometric depiction of a triangle matroid.
In our construction, the points on each side of the triangle correspond with the elements of the generating set $S$ and their inverses. They are marked by an upper index to distinguish between different sides (see the figure above).

A $c$-arrangement representation of the triangle matroid corresponds to a block matrix with blocks of size $c\times c$, having three block rows and a block column for each element of the matroid. 
This block matrix can be brought into a normal form such that an element $a^{(i)}$ for $1\le i\le 3$ has the blocks $-I_c, A_i,0$ in the rows $i,i+1,i+2$ respectively (regarding the indices cyclically modulo three) where $I_c$ is the identity matrix and $A_i\in GL_c(\F)$.
Furthermore, we add circuits to the matroid $N_{S,R}$  to guarantee that the matrices corresponding to $x^{(1)},x^{(2)},x^{(3)}$ are all equal to the same matrix $X\in GL_c(\F)$. 
\Cref{lem:block_matrices} then shows that the elements $x^{(1)},y^{(2)},z^{(3)}$ form a circuit in the matroid if and only if their corresponding matrices satisfy $YX=Z^{-1}$.

Hence, the matrices in $c$-arrangement representations of the matroid depicted in \Cref{fig:triangle_matroid} satisfy $YX=Z^{-1}$ and $XY=W^{-1}$.
By construction, $z^{(3)}$ and $w^{(3)}$ are different elements in the matroid which means that the matroid is only representable if there are $X,Y\in GL_c(\F)$ such that $XY\neq YX$.
This implies that the matroid can only be representable by $c$-arrangements for $c\ge 2$.
This reflects the result of Goresky and MacPherson that the non-Pappus matroid is representable as $c$-arrangement for $c\ge 2$~\cite[p. 257]{GM88}.

 In later sections, we manipulate this matroid and exhibit a family of matroids such that at least one of them is representable as a $c$-arrangement if and only if the given UWPFG instance has a negative answer.

Any group representation of $\langle S\mid R\rangle$ in $GL_c(\F)$ gives rise to a subspace arrangement where each subspace is of dimension $c$. Such arrangements often do not represent the matroid $N_{S,R}$ as $c$-arrangements, and we call them \emph{weak $c$-representations}.
The issue already arises in the matroid depicted in \Cref{fig:triangle_matroid}:
the sum of the subspaces corresponding to the elements $x^{(1)}$ and $y^{(1)}$ is of dimension $2c$ if and only if $X-Y$ is invertible.
This is not the case for all pairs of non-commuting matrices.

\begin{figure}[b]
	\begin{subfigure}[b]{0.4\linewidth}
		\centering
		\includegraphics[scale=.6]{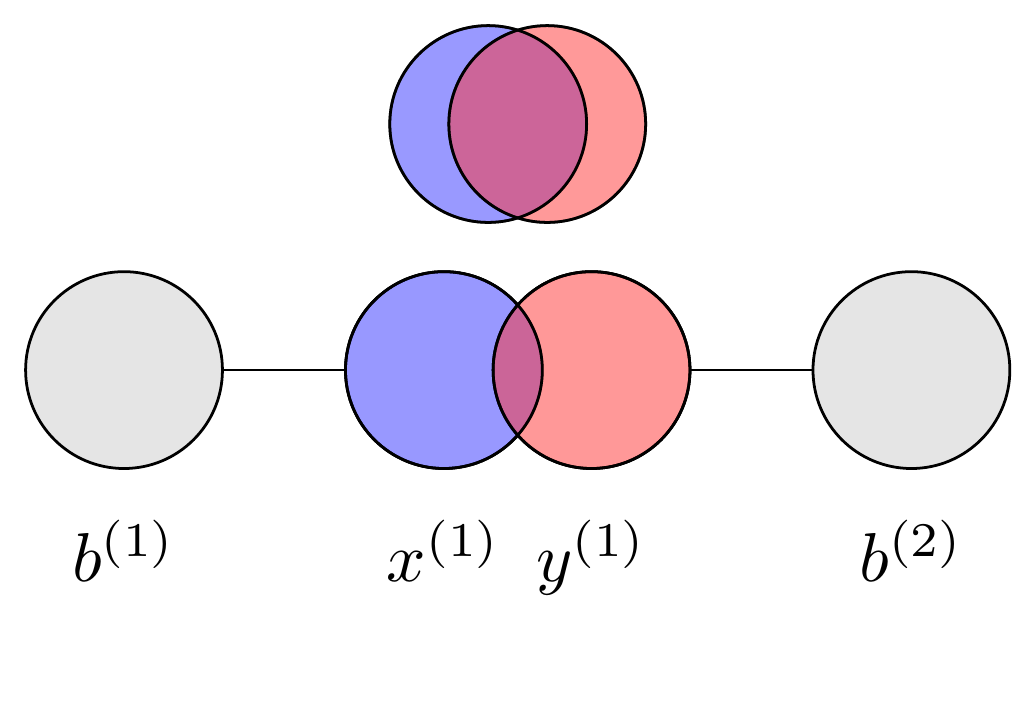}
		\caption{Inflation of two elements.}
		\label{fig:regularization_pair}
	\end{subfigure}
	\begin{subfigure}[b]{0.6\linewidth}
		\centering
		\includegraphics[scale=.6]{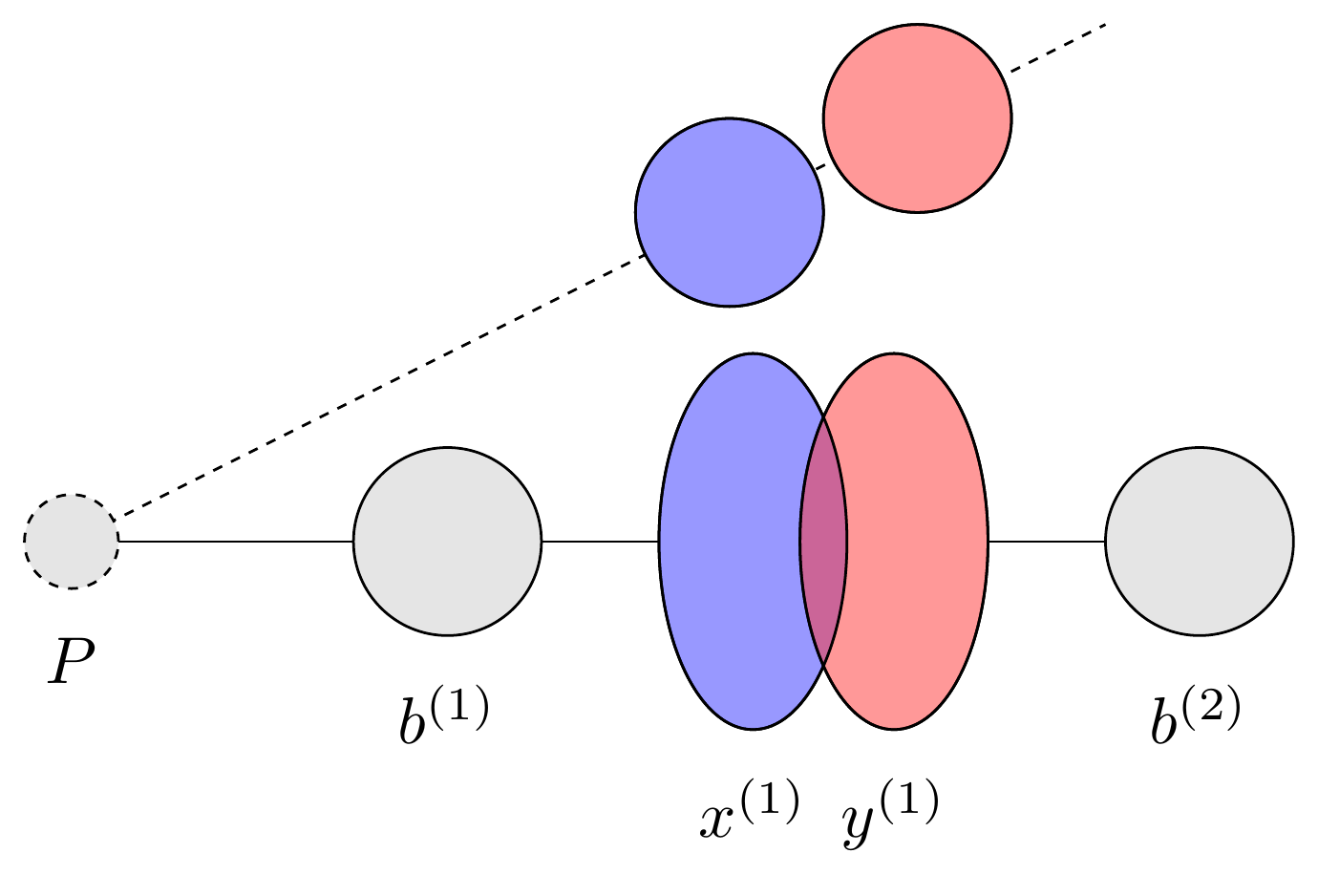}
		\caption{Infaltion with respect to the basis.}
		\label{fig:regularization_basis}
	\end{subfigure}
	\caption{The two inflation steps as explained in the last paragraph.
		The colored shapes above the line through $b^{(1)}$ and $b^{(2)}$ are the added subspaces in both figures.}
	\label{fig:regularization}
\end{figure}

To overcome this hurdle we develop an algebraic construction, \emph{inflation}, in \Cref{sec:alg_reg}.
As an example consider the subspaces of $x^{(1)}$ and $y^{(1)}$ which have a non-trivial intersection as depicted in \Cref{fig:regularization_pair}.
We add new subspaces to each of these subspaces such that their intersection is of dimension $c$.
To ensure that the sum of the subspaces corresponding to $x^{(1)}$ and $y^{(1)}$ intersects the subspace corresponding to the sum of $b^{(1)}$ and $b^{(2)}$ in a subspace of dimension a multiple of $c$, we add two more subspaces to $x^{(1)}$ and $y^{(1)}$ that lie in on a line which intersects the line spanned by $b^{(1)}$ and $b^{(2)}$ non-trivially.
This second step is depicted in \Cref{fig:regularization_basis}. The arrangement resulting from this procedure depends on the original arrangement, but its rank function does not.
Furthermore, the resulting full arrangement contains enough information to reconstruct the original one.

\Cref{sec:comb_reg} carries out an analogous combinatorial construction for polymatroids.

In \Cref{sec:comp} we show that applying sufficiently many inflation steps, one can obtain a $c$-arrangement. Along the way, we show the combinatorial and algebraic inflation procedures are compatible. This compatibility is central to the rest of the proof.

In \Cref{sec:c-bases}, we perform the final technical construction necessary to relate these inflated polymatroids to the UWPFG: we show that, from a inflation of $N_{S,R}$, one can produce a finite set of matroids such that at least one is representable as a $c$-arrangement if and only if the UWPFG has a negative answer. The key to relating inflations of $N_{S,R}$ to the UWPFG is being able to detect combinatorially whether a certain pair of subspaces is equal, and this is the main goal of the section.
We conclude by proving the main theorem in \Cref{sec:main_proof}.

\section{Preliminaries}\label{sec:pre}

In this section we collect definitions and basic properties which will be used throughout the article.
We also need some elementary matroid theory, but we will not cover it here.
We recommend the beginning of Oxley's book~\cite{Oxl11}.

\subsection{Subspace Arrangements}
\begin{defn}\label{def:arr}
	Let $V$ be a vector space over a field $\F$ and let $E$ be a finite set.
	A \emph{subspace arrangement} $\A$ is a set of subspaces $\{A_e\}_{e\in E}$ where $A_e$ is a subspace of $V$ for each $e\in E$.
	For a subspace arrangement $\A$ we will  use the notation $A_X\coloneqq \sum_{e\in X}A_e$ for any subset $X\subseteq E$. Further,
	\begin{enumerate}
		\item We call $\A$ $c$-\emph{homogeneous} for some $c\ge 1$ if $\dim A_e= c$ for all $e\in E$.
		\item We call $\A$ $c$-\emph{admissible} for some $c\ge 1$ if for any subset $X\subseteq \A$ the dimension of $A_X$ is a multiple of $c$.
		\item A $c$-arrangement is a subspace arrangement which is both $c$-homogeneous and $c$-admissible.
	\end{enumerate}
\end{defn}

The notion of a $c$-admissible subspace arrangement is not standard, but will be useful in \Cref{sec:c-bases}.

\begin{rmrk}
 Note that $c$-arrangements are often equivalently defined as subspace arrangements such that all \emph{intersections} are of \emph{codimension} a multiple of $c$.
 We prefer to work in the dual framework, with subspaces of dimension $c$ and sums instead of intersections. This duality is the same as that between matroid representations and hyperplane arrangements: it takes a subspace of a vector space $V$ to its annihilator in the dual $V^\ast$. Taking a sum of subspaces is dual to taking their intersection.
\end{rmrk}

\begin{defn}\label{def:rank}
	To a subspace arrangement $\A=\{A_e\}_{e\in E}$ over the field $\F$ we define two \emph{rank functions} on the power set $\mathcal{P}(E) $. Fix a subset $X\subseteq E$.
	\begin{enumerate}
		\item The usual rank function $r_{\A} : \mathcal{P}(E) \rightarrow \N$ is defined by  $r_{\A}(X)\coloneqq\dim (A_X)$.
		\item For any $c\ge 1$ we define the \emph{normalized} rank function $r_{\A}^c : \mathcal{P}(E) \rightarrow \Q$  by setting $r_{\A}^c(X)\coloneqq \frac{1}{c}\dim (A_X)$.
	\end{enumerate}
\end{defn}

\begin{rmrk}
	A $c$-homogeneous subspace arrangement $\A$ is a $c$-arrangement if and only if its normalized rank function $r_{\A}^c$ takes only integral values.
\end{rmrk}

Now we can define when a matroid is representable by a $c$-arrangement.
In our proofs we will additionally need a weaker notion of representations, due to the issues arising from subspaces having non-trivial intersections as discussed in \Cref{sec:outline}.

\begin{defn}\label{def:weak}
	Fix a matroid $M$ on the ground set $E$ with rank function $r$.
	\begin{enumerate}
		\item A matroid $M=(E,r)$ is called \emph{multilinear} of order $c$ over a field $\F$ if there exists a $c$-arrangement $\A=\{A_e\}_{e\in E}$ such that their (normalized) rank functions agree, i.e. $r(X)=r_{\A}^c(X)$ for any $X\subseteq E$.
		We say that the $c$-arrangement $\A$ \emph{represents} the matroid $M$ in that case.
		\item To define a weaker representability notion, we fix a basis $B$ of the matroid $M$.
		 We say that a $c$-homogeneous subspace arrangement $\A=\{A_e\}_{e\in E}$ \emph{weakly represents} $M$ with respect to the basis $B$ if 
		$r(X)\ge r_{\A}^c(X)$ for all subsets $X\subseteq E$ and $r(Y)= r_{\A}^c(Y)$ for all subsets $Y\subseteq E$ with $|Y\setminus B|\le 1$.
		In this case, we also say $\A$ is a weak $c$-representation of $M$.
	\end{enumerate}
\end{defn}

Our proofs often use the following dimension formula without explicitly mentioning it:
If $U_1,U_2$ are finite dimensional subspaces of a vector space $V$ then
\[
\dim(U_1+U_2) = \dim (U_1)+\dim(U_2) - \dim(U_1\cap U_2).
\]

\subsection{Linear-algebraic Calculations}
The following three lemmas are collected here to avoid cluttering later sections. The first is a basic tool in \cref{sec:staudt}. The other two will be used to bound the intersections of certain subspaces in \cref{sec:well_sep}.

\begin{lemm}\label{lem:block_matrices}
	Let $\F$ be a field and $A,B,C\in M_k(\F)$ any $k\times k$ matrices. 
	\begin{enumerate}
		\item\label{it:AB} The block matrix $
		\begin{bsmallmatrix}
		-I_k & -I_k\\
		A& B\\
		0 & 0
		\end{bsmallmatrix}$
		has rank $k+\rk(B-A)$.
		\item\label{it:ABC} The block matrix $
		\begin{bsmallmatrix}
		-I_k & 0    &  C\\
		A    & -I_k &  0\\
		0    & B     & -I_k 
		\end{bsmallmatrix}$
		has rank $2k+\rk(BAC-I_k)$.			
	\end{enumerate}
\end{lemm}
\begin{proof}
	In the case of \eqref{it:AB}, we multiply the matrix from the right with the invertible block matrix $\begin{bsmallmatrix}
	I_k & -I_k  \\
	0  & I_k \\
	\end{bsmallmatrix}$ which preserves its rank. 
	Thus, we obtain the block matrix $\begin{bsmallmatrix}
	-I_k & 0\\
	A    & B-A \\
	0    & 0
	\end{bsmallmatrix}$ which immediately implies the claim on its rank.
	
	Analogously for the case \eqref{it:ABC}, we multiply the matrix from the right with the invertible block matrix $\begin{bsmallmatrix}
	I_k & 0  &  C\\
	0  & I_k &  AC\\
	0  & 0  & I_k 
	\end{bsmallmatrix}$ which preserves its rank.
	Hence, we obtain the block matrix $\begin{bsmallmatrix}
	-I_k & 0    &  0\\
	A    & -I_k &  0\\
	0    & B     & BAC-I_k 
	\end{bsmallmatrix}$ which finishes the proof.
\end{proof}

\begin{lemm}\label{lem:rank_permuatation_matrix}
	Let $\sigma\in S_k$ be a permutation with no fixed points, i.e.\ a derangement, and $P_{\sigma}$ the corresponding $k\times k$ permutation matrix over some field $\F$.
	Then $\rk(P_{\sigma}-I_k)\ge \frac{k}{2}$.
\end{lemm}
\begin{proof}
	Consider the graph $G_{\sigma}$ with vertices $\{1,\dots,k\}$ having an an edge $\{i,j\}$ if $\sigma(i)=j$.
	The matrix $P_{\sigma}-I_k$ is a representation matrix of the graphic matroid $M(G_{\sigma})$  over $\F$ (cf.~\cite[Lemma 5.1.3]{Oxl11}).
	
	Suppose $\sigma$ has a decomposition into $r$ disjoint cycles.
	This implies the graph $G_{\sigma}$ consists of $r$ disjoint cycles, so $M(G_{\sigma})$ is a direct sum of $r$ circuits.
	Therefore, we have
	\[
	\rk(P_{\sigma}-I_k)=\rk (M(G_{\sigma}))= k-r.
	\]
	By assumption $\sigma$ is a derangement which means that each cycle has length at least two.
	The result now follows from the fact that $\sigma$ has at most $\frac{k}{2}$ cycles.
\end{proof}

\begin{coro}\label{lem:rank_reg_rep}
	Let $G$ be a finite group and let $\{P_g\}_{g\in G}$ be the permutation matrices of its regular representation (these are the permutation matrices of $G$'s action on itself by left multiplication). Then for any distinct $g_1,g_2\in G$: \[\rk(P_{g_1}-P_{g_2}) \ge \frac{|G|}{2}.\]
\end{coro}
\begin{proof}
	Note that $P_g$ is the permutation matrix of a derangement for any $g\in G$ other than $e$: otherwise, in the action of $G$ on itself by left multiplication, $g$ has a fixed point, say $h$, and $gh=h$; but this implies $g=e$.
	
	For distinct $g_1,g_2 \in G$ \[\rk(P_{g_1}-P_{g_2}) = \rk((P_{g_1}-P_{g_2})P_{g_2}^{-1}),\] since $P_{g_2}^{-1}$ is invertible.
	Thus \cref{lem:rank_permuatation_matrix} implies 
	\[
	\rk(P_{g_1}-P_{g_2}) = \rk(P_{g_1 g_2^{-1}} - I_{|G|}) \ge \frac{|G|}{2.} \qedhere
	\]
\end{proof}

\subsection{Group Presentations}

We collect basic definitions for group presentations and refer to the book by Lyndon and Schupp for details~\cite{LS77}. 

\begin{defn}
	Let $\langle S\mid R \rangle$ be a finite presentation, that is $S$ is a finite list of symbols and $R$ is a finite set of words in $S$ and their inverses.
	Let $F_S$ be the associated free group generated by $S$, and let $N$ be the normal closure of the subgroup generated by $R$ in $F_S$.
	We define the group $G_{S,R}\coloneqq F_S / N$.
	Two presentations $\langle S\mid R \rangle$ and $\langle S'\mid R' \rangle$ are \emph{equivalent}  if the groups $G_{S,R}$ and $G_{S',R'}$ are isomorphic.
\end{defn}

To simplify the constructions below we will restrict ourselves to presentations with relations of length three. Using Tietze transformations, one can reduce to this situation from the general case. We state this in a lemma.

\begin{lemm}\label{lem:three_relations}
	Any finite presentation of a group $\langle S\mid R \rangle$ is equivalent to a finite presentation $\langle S'\mid R' \rangle$ where $R'$ consists of relations of length three only.
	Moreover, any instance of a UWPFG $\langle S\mid R \rangle,w$ can be equivalently reformulated such that any word in $R$ is of length three and $w\in S$.
\end{lemm}

\subsection{Characteristic of the Field}

Rado proved that a matroid which is representable as a $1$-arrangement over a field $\F$, that is linearly representable in the usual sense, is representable over a finite algebraic extension over the prime field of $\F$~\cite{Rad57}.
The proof uses Hilbert's Nullstellensatz and directly translates to the situation of $c$-arrangement representations.
\begin{prop}\label{prop:field}
	Let $\A$ be $c$-arrangement over a field of characteristic $p\in\mathbb{P}\cup \{0\}$ representing a matroid $M$.
	Then, $M$ has a $c'$-arrangement $\A'$ representation over the prime field of characteristic $p$.
\end{prop}
\begin{proof}
	The discussion above shows that we can assume that $\A$ is a $c$-arrangement in a vector space $V$ over a field $\mathbb{L}$ which is a finite extension of degree $d$ over its prime field $\F$.
	Since $V$ is also a vector space over $\mathbb{F}$ and $\dim_{\mathbb{F}}(U)=d\dim_{\mathbb{L}} (U)$ for any finite-dimensional subspace of $V$, the arrangement $\A$ is naturally a $(c\cdot d)$-arrangement over the prime field $\mathbb{F}$.
	Their normalized rank functions agree.
\end{proof}
Using this proposition  we may assume to work over an arbitrary infinite field.
All following proofs are not dependent on the characteristic of the field.
We mention that the proofs also work for a subspace representation over a specified set of characteristics.

\subsection{Genericity}
We will often choose \emph{generic} subspaces of a given vector space. Essentially, a generic subspace is one which does not satisfy some given collection of Zariski-closed conditions. 

Consider for example the following situation: we are given a vector space $V$ and a subspace $W$ of $V$. If $U$ is of dimension complementary to $W$ in $V$, we expect $W+U=V$; this is a generic condition on $U$, in the sense that it is satisfied on a dense Zariski-open subset of the subspaces of appropriate dimension.

The example is almost as general as we need here. When we say $U$ is chosen generically from some family $\mathcal{F}$ of subspaces of $V$, we mean the following: for each $W$ in some family $\mathcal{G}$ of subspaces of $V$, \[\dim(U+W)=\max\{\dim(U'+W)\mid U'\in\mathcal{F}\}.\]

In this paper, the family $\mathcal{F}$ is always the family of $d$-dimensional subspaces of some fixed subspace of $V$. The family $\mathcal{G}$ is not specified explicitly, but there is always a suitable finite $\mathcal{G}$. Under these conditions, there is always a $U\in\mathcal{F}$ fulfilling the genericity condition as long as the field is infinite.

\section{Generalized Dowling Geometries}\label{sec:staudt}

In this section we reduce the uniform word problem for finite groups (UWPFG) to a problem on weak $c$-arrangements: given an instance of a UWPFG, we construct a matroid~$M$ such that the given instance has a negative answer if and only if for some $c$, $M$ has a weak $c$-representation which satisfies an additional condition.

To do this we encode group presentations in matroids which we call generalized Dowling geometries. These matroids all have a special form which simplifies many calculations, and which we describe first.

\begin{defn}\label{def:triangle_matroid}
	We call a matroid $M=(E,r)$ a \emph{triangle matroid} if it is of rank three and there exists a basis $B=\{b^{(1)},b^{(2)},b^{(3)}\}$ such that all elements of $E$ are contained in the flats spanned by $\{b^{(1)},b^{(2)}\}$, $\{b^{(1)},b^{(3)}\}$ or $\{b^{(2)},b^{(3)}\}$.
	To ease our notation we call
	\begin{enumerate}
		\item the elements in $B$ the \emph{vertices} of the triangle,
		\item the flats $\{b^{(1)},b^{(2)}\}$, $\{b^{(1)},b^{(3)}\}$ or $\{b^{(2)},b^{(3)}\}$ the \emph{sides} of the triangle, and
		\item the elements of $E$ \emph{bottom, left, and right} elements if they lie on the flats $\{b^{(1)},b^{(2)}\}$, $\{b^{(1)},b^{(3)}\}$, and $\{b^{(2)},b^{(3)}\}$ respectively.
	\end{enumerate}
	Additionally, for any subset $S\subseteq E$, we denote by $C_M(S)$ a subset in $B$ such that $r(S)=r(S\cup C_M(S))=r(C_M(S))$.
	In other words, $C_M(S)$ is a subset of $B$ such that the closures of $S$ and of $C_M(S)$ are equal.
	If it exists $C_M(S)$ is unique. Otherwise, we set $C_M(S)\coloneqq \emptyset$.
\end{defn}
\begin{rmrk}
	A triangle matroid is the same as a \emph{frame matroid} of rank three as introduced by Zaslavsky~\cite{Zas94}.
	We will not use any non-trivial property of frame matroids.
\end{rmrk}

\Cref{fig:triangle_matroid} depicts a geometric representation of a triangle matroid.
In the following we construct generalized Dowling geometries from group presentations.

\begin{defn}[Generalized Dowling Geometry]\label{def:matroid_N}
	Let $\langle S\mid R \rangle $ be a finite presentation of a group.
	By \Cref{lem:three_relations} we can assume that any relation in $R$ is of length three.
	
	We construct a triangle matroid $N_{S,R}$ on the ground set $E_{S,R}$ with basis $B = \{b^{(1)}, b^{(2)}, b^{(3)}\}$ by describing its dependent flats $\mathcal{F}_{S,R}$ of rank $2$, where we regard the indices cyclically modulo $3$:
	\begin{align*}
E_{S,R} \coloneqq 				   & \{b^{(i)}, e^{(i)}, x^{(i)}, x^{-1^{(i)}}\mid 1\le i \le 3\mbox{ and }x\in S  \}, \\
\mathcal{F}_{S,R} \coloneqq & \{\bigcup_{x\in S}\{x^{(i)}, x^{-1^{(i)}}\}\cup \{e^{(i)},b^{(i)},b^{(i+1)}\}\mid  \mbox{ for any fixed }1\le i \le 3\} \,\cup \\
& \{ \{x^{(i)},x^{-1^{(j)}},e^{(k)}\}\mid \mbox{ for } x\in S\mbox{ and pairwise different  } 1\le i,j,k\le 3 \}\, \cup \\
& \{\{e^{(1)},e^{(2)},e^{(3)}\}\} \, \cup \{  \{x^{(2)},y^{(1)},z^{(3)}\} \mid \mbox{ for any } xyz \in R \}.
\end{align*}
\end{defn}
This defines a unique matroid of rank $3$ with basis $B$, since $B$ is not contained in any of these subsets, and any two distinct such rank $2$ flats intersect in at most one element (cf.~\cite[Proposition 1.5.6]{Oxl11}).

To relate linear (group) representations of $\langle S\mid R \rangle $ to the matroid $N_{S,R}$, we investigate its weak $c$-representations with respect to the basis $B$ as defined in \Cref{def:weak} (b).
To work with such representations, we will regard them as block matrices of size $3c\times |E_{S,R}|c$ with blocks of size $c\times c$.
The block columns are indexed by the elements of $E_{S,R}$, and the $c$ columns in each block column are a basis for the corresponding subspace in a weak representation.

Recall that the \emph{regular representation} of a finite group $G$ with $n\coloneqq |G|$ over any field $\F$ is a linear representation $\rho : G \mapsto GL_n(\F)$ induced by the action of each element $g\in G$ on $G$ itself by left multiplication.

\begin{prop}\label{pro:G_weak_representation}
	Let $\langle S\mid R \rangle $ be a finite presentation of a group with relations of length three and let $G$ be a finite group with a homomorphism $\phi: G_{S,R}\rightarrow G$.
	Set $n\coloneqq |G|$ and fix any field $\F$.
	Let $\rho: G \rightarrow GL_n(\F)$ be the regular representation of $G$.
	Then the $n$-homogeneous subspace arrangement $\A_{G,\phi}$ over $\F$ given by the $3n\times |E_{S,R}|n$ block matrix~$A_{G,\phi}$ is a weak $n$-representation of the matroid $N_{S,R}$ with respect to the basis $B$:
		\[ A_{G,\phi}\coloneqq
	\begin{bheadmatrix}[b^{(1)} & b^{(2)}   	& b^{(3)} 			      &  x^{(1)} 								   & 			 & x^{(2)}			 	   					  & 	       & x^{(3)} 				 		  	 & 			][*{9}{r}]
	I_n \hphantom{0)} & 0 \hphantom{0)} 	& 0 \hphantom{0)}   & -I_n \hphantom{0)} 			   & \cdots &  0 \hphantom{0)} 					 & \cdots & \rho(\phi(x)) \hphantom{0)}	  & \cdots \\
	0 \hphantom{0)}    & I_n \hphantom{0)}  & 0 \hphantom{0)}   &  \rho(\phi(x)) \hphantom{0)} & \cdots & -I_n \hphantom{0)} 				& \cdots &  0         \hphantom{0)} 	   & \cdots \\
	0 \hphantom{0)}    & 0 \hphantom{0)} 	& I_n\hphantom{0)}  &  0 \hphantom{0)} 					 & \cdots &  \rho(\phi(x))  \hphantom{0)} & \cdots & -I_n         \hphantom{0)} & \cdots
	\end{bheadmatrix}.
	\]
\end{prop}
\begin{proof}
	Each block column of $A_{G,\phi}$ of size $3n\times n$ is indexed by an element of the matroid~$N_{S,R}$.
	This correspondence defines a subspace arrangement $\A=\{A_e\}_{e\in E_{S,R}}$ where the subspace $A_e$ is spanned by the column vectors of the block column $e$.
	So in particular we have 
	\[ 
		A_{x^{(1)}}=\spa
		\begin{bsmallmatrix}
			-I_n \\
			\rho(\phi(x)) \\
			0   
		\end{bsmallmatrix}, \quad
		A_{x^{(2)}}= \spa
		\begin{bsmallmatrix}
			0\\
			-I_n \\
			\rho(\phi(x)) \\
		\end{bsmallmatrix} \;\mbox{  and  }\;
		A_{x^{(3)}}=\spa
		\begin{bsmallmatrix}
			\rho(\phi(x)) \\
			0 \\
			-I_n
		\end{bsmallmatrix},
	\]
	for any $x\in S$ or $x^{-1}\in S$.
	The normalized rank $r^n_{A_{G,\phi}}$ of any subset of elements in $E_{S,R}$ containing only bottom, right or left elements is at most $2$ since it involves at most two non-zero blocks by definition of the matrix $A_{G,\phi}$.

	Next consider any three elements $x^{(2)},y^{(1)},z^{(3)}$ of $E_{S,R}$ with $x,y,z$ elements of $S$, inverses of elements of $S$ or neutral elements in $G_{S,R}$.
	By \Cref{lem:block_matrices} b) we have
	\begin{align*}
	r^n_{\A_{G,\phi}}( \{ x^{(2)},y^{(1)},z^{(3)}\} ) &= 2 + \frac{1}{n} \rk (\rho(\phi(x)) \rho(\phi(y)) \rho(\phi(z)) -I_n )\\
	& = 2 + \frac{1}{n} \rk (\rho(\phi(xyz-e))).
	\end{align*}
	This implies that any circuit of $N_{S,R}$ of the form $\{  x^{(2)},y^{(1)},z^{(3)} \}$ for a relation $xyz \in R$ corresponds to a subset of $\A_{G,\phi}$ of normalized rank $2$ since we have $xyz=e$ in $G_{S,R}$ in this case.
	The same argument holds for circuits of the form $\{ x^{(i)},x^{-1^{(j)}},e^{(k)}\}$ for any $x\in S$ and pairwise different indices $1\le i,j,k \le 3$.

	Clearly we have $r^n_{\A_{G,\phi}}(\{a\})=1$ for any element $a\in E_{S,R}$.
	Lastly, consider any subset of the form $\{b^{(i)},x^{(j)}\}$ for indices $1\le i,j\le 3$ and $x\in S$ or $x^{-1}\in S$.
	By symmetry we can without loss of generality assume $i=1$ and $j=1$.
	Hence, this subset corresponds to the block matrix 
	$\begin{bsmallmatrix}
	I_n & -I_n\\
	0    &  \rho(\phi(x)) \\
	0    & 0
	\end{bsmallmatrix}$
	which is of rank $2n$ since the matrix $\rho(\phi(x)) $ is by assumption invertible.
	Thus, we obtain $r^n_{\A_{G,\phi}}(\{b^{(1)},x^{(1)}\})=2$.
	The cases of subsets containing two or three elements of the basis $\{b^{(1)},b^{(2)},b^{(3)}\}$ and one additional element can be checked in the same way.
	This completes the proof that $\A_{G,\phi}$ is a weak representation of the matroid $N_{S,R}$ with respect to the basis $\{b^{(1)},b^{(2)},b^{(3)}\}$.
\end{proof}
 
 Conversely, the next proposition shows how to obtain a group from a weak representation of $N_{S,R}$.

\begin{prop}\label{pro:weak_representation_G}
	Let again be $\langle S\mid R \rangle $ be a finite presentation of a group with relations of length three and consider the matroid $N_{S,R}$.
	Any weak $c$-representation $\A$ of $N_{S,R}$ with respect to the basis $\{b^{(1)},b^{(2)},b^{(3)}\}$ over a field $\F$ yields a group $G_{\A}$ that is a finitely generated subgroup of $GL_c(\F)$ and a group homomorphism $\phi_{\A}:G_{S,R}\rightarrow G_{\A}$.
\end{prop}
\begin{proof}
	Let $\A=\{A_e\}_{e\in E_{S,R}}$ be a weak $c$-representation of the matroid $N_{S,R}$ with respect to the basis $\{b^{(1)},b^{(2)},b^{(3)}\}$ over a field $\F$.
	The arrangements $\A$ yields a $3c \times |E_{S,R}|c$ block matrix over $\F$ where each $3c\times c$ block column is indexed by an element of the matroid $N_{S,R}$ and contains a basis of the corresponding subspace in $\A$.
	
	We perform the following invertible operations which preserve the underlying combinatorial structure:
	\begin{enumerate}
		\item After a change of coordinates of the ambient vector space $\F^{3c}$, we can assume that the matroid basis $\{b^{(1)},b^{(2)},b^{(3)}\}$ is represented by the block matrix 
		$\begin{bsmallmatrix}
		I_c & 0 & 0\\
		0    &  I_c& 0 \\
		0    & 0 & I_c
		\end{bsmallmatrix}$.
		This can be accommodated by multiplying the entire matrix by an invertible matrix of size $3c$ from the right.
		\item Consider $x^{(1)}$ any bottom element of the matroid $N_{S,R}$.
		Since $r_{S,R}(\{b^{(1)},x^{(1)},b^{(2)}\})=r_{S,R}(\{b^{(1)},x^{(1)}\})=r_{S,R}(\{x^{(1)},b^{(2)}\})=2$ where $r_{S,R}$ is the rank function of the matroid $N_{S,R}$, the normalized rank $r^c_{\A}$ of the corresponding sets of subspaces of $\A$ must be $2$ as well.
		Hence, the block column of $x^{(1)}$ is of the form
		 $\begin{bsmallmatrix}
		X_1' \\
		X_1''  \\
		0
		\end{bsmallmatrix}$ where $X',X''$ are invertible $c\times c$ matrices.
		Analogous arguments show that the block columns of any right and left elements of $N_{S,R}$ are of the form
		$\begin{bsmallmatrix}
		0 \\
		X_2' \\
		X_2''  
		\end{bsmallmatrix}$ and $
		\begin{bsmallmatrix}
		X_3'  \\
		0 \\
		X_3''
		\end{bsmallmatrix}$ for suitable invertible $c\times c$ matrices $X_2',X_2'',X_3',X_3''$ respectively.
		\item After multiplying the block column $e^{(1)}$ and the second block row from the right with an invertible matrix of size $c$ we can assume that the block column $e^{(1)}$ is
		$\begin{bsmallmatrix}
		-I_c \\
		I_c \\
		0
		\end{bsmallmatrix}$.
		Similarly after multiplying the block column $e^{(2)}$ and the third block row, we can assume that the block column $e^{(2)}$ is $\begin{bsmallmatrix}
		0 \\
		-I_c \\
		I_c 
		\end{bsmallmatrix}$.
		Subsequently, we perform a multiplication from the right on the block column $e^{(3)}$ after which it is of the form $\begin{bsmallmatrix}
		E_3 \\
		0 \\
		-I_c 
		\end{bsmallmatrix}$ for some invertible $c\times c$ matrix $E_3$.
		Since $\{e^{(1)},e^{(2)},e^{(3)}\}$ is a circuit of $N_{S,R}$ \Cref{lem:block_matrices} b) implies that $E_3-I_c=0$ which implies $E_3=I_c$.
		\item Lastly, by multiplying every block column again by a suitable invertible matrix of size $c$ from the right we can assume that the block matrix defining $\A$ is of the following form where $T_{x^{(i)}}$ is an invertible matrices for any $x\in S$ and $1\le i \le 3$:
		\[ 
		\begin{bheadmatrix}[b^{(1)} & b^{(2)}   	& b^{(3)} 			      &  e^{(1)} &	x^{(1)} 				& 	        & e^{(2)} & x^{(2)}			 	   	  & 	       &  e^{(3)} &x^{(3)} 				 & 			][*{12}{r}]
		I_c \hphantom{0)} & 0 \hphantom{0)} 	& 0 \hphantom{0)}   &-I_c \hphantom{0)}  & -I_c \hphantom{0)} 	& \cdots & 0 \hphantom{0)}      & 0 \hphantom{0)} 		& \cdots & I_c \hphantom{0)}    & T_{x^{(3)} }  \hphantom{0)}  & \cdots \\
		0 \hphantom{0)}    & I_c \hphantom{0)}  & 0 \hphantom{0)}   &  I_c \hphantom{0)}  & T_{x^{(1)} } \hphantom{0)}  & \cdots & -I_c \hphantom{0)}  & -I_c \hphantom{0)}  & \cdots & 0 \hphantom{0)}      &  0     \hphantom{0)} & \cdots \\
		0 \hphantom{0)}    & 0 \hphantom{0)} 	& I_c\hphantom{0)}  &  0 \hphantom{0)}    & 0 \hphantom{0)} 	  & \cdots &  I_c \hphantom{0)}   & T_{x^{(2)} } \hphantom{0)}    & \cdots & -I_c \hphantom{0)}  & -I_c \hphantom{0)} & \cdots
		\end{bheadmatrix}.
		\]
	\end{enumerate}

	Now consider the elements $x^{(1)},x^{(2)},x^{(3)},x^{-1^{(1)}},x^{-1^{(2)}},x^{-1^{(3)}}$:
	The circuits of the form $\{ x^{(i)},x^{-1^{(j)}},e^{(k)}\}$ for pairwise different  $ 1\le i,j,k\le 3$ in connection with \Cref{lem:block_matrices} b) imply  $T_{x^{(i)}}^{-1}=T_{x^{-1^{(j)}}}$ for any $i\neq j$.
	A second application of the same lemma implies $T_{x^{(i)}}=T_{x^{(j)}}$ and $T_{x^{-1^{{(i)}}}}=T_{x^{-1^{(j)}}}$ for all $i\neq j$.
	
	For a relation $xyz \in R$ the elements $\{  x^{(2)},y^{(1)},z^{(3)}\}$ form a circuit of the matroid $N_{S,R}$.
	Since the arrangement $\A$ is a weak $c$-representation of $N_{S,R}$ the normalized rank of the corresponding subspaces is at most $2$ which implies that the block matrix
	$\begin{bsmallmatrix}
	0 	   & -I_c  &   T_{z^{(3)}}\\
	-I_c  & T_{y^{(1)}}  & 0\\
	T_{x^{(2)}}  &  0 	   & -I_c
	\end{bsmallmatrix}$
	is of rank at most $2c$.
	Thus, \Cref{lem:block_matrices} b) implies 
	\begin{equation}\label{eq:matrices_relation}
		T_{x^{(2)}} T_{y^{(1)}} T_{z^{(3)}} = I_c.
	\end{equation}
	
	Now we set  the group $G_{\A}$ to be generated by the matrices $T_{x^{(1)}}$ for all $x\in S$.
	Hence, $G_{\A}$ is a finitely generated subgroup of $GL_c(\F)$.
	We define the group homomorphism $\phi_{\A}:G_{S,R}\rightarrow G_{\A}$ by setting $\phi_{\A}(x) \coloneqq T_{x^{(1)}}$ for any $x\in S$.
	\Cref{eq:matrices_relation} implies that this homomorphism is well-defined, i.e. respects the relations $R$ of $G_{S,R}$.
\end{proof}

In the following theorem we establish the connection between the UWPFG and weak $c$-representations.
\Cref{lem:three_relations} allows to assume without loss of generality that the relations are of length three and the word $w$ is an element of $S$.

\begin{theorem}\label{theo:weak_undecidable}
	Consider a UWPFG instance given by finite presentation $\langle S\mid R \rangle $  and an element $w\in S$.
	Then, the answer to this instance is negative, i.e.\ there exists a finite group $G$ with a homomorphism $\phi: G_{S,R}\rightarrow G$ and $\phi(w)\neq e_G$, if and only if there exists a weak $c$-representation $\A=\{A_e\}_{e\in E_{S,R}}$ over a field $\F$ of the matroid $N_{S,R}$ with respect to the basis $\{b^{(1)},b^{(2)},b^{(3)}\}$ such that
	\begin{equation}\label{eq:rank_2}
	r^c_{\A}( \{w^{(1)},e^{(1)}\}) > 1.
	\end{equation}
\end{theorem}
\begin{proof}
	First, assume that there exists a finite group $G$ with a homomorphism $\phi: G_{S,R}\rightarrow G$ and $\phi(w)\neq e_G$.
	Set $n\coloneqq |G|$.
	\Cref{pro:G_weak_representation} shows that there exists a $n$-homogeneous subspace arrangement $\A_{G,\phi}=\{A_e\}_{e\in E_{S,R}}$ over any field $\F$ induced by the regular representation $\rho: G \rightarrow GL_n(\F)$ weakly representing $N_{S,R}$.
	To investigate the word $ w\in S$ in $G$ we can compute using \Cref{lem:block_matrices} a)
	\begin{align*}
		r^n_{\A_{G,\phi}}( \{ w^{(1)},e^{(1)}\}) &= \frac{1}{n} \rk\begin{bsmallmatrix}
			-I_n	   & -I_n   \\
			\rho(\phi(w)) & \rho(\phi(e))  \\
			0 &  0 	  
		\end{bsmallmatrix}\\
		& = 1 + \frac{1}{n} \rk (\rho(\phi(w))-I_n).
	\end{align*}
	The assumption $\phi(w)\neq e_G$ implies $\rho(\phi(w))\neq I_n$.
	Therefore, \Cref{eq:rank_2} holds.
	
	Conversely, assume that there exists a weak $c$-representation $\A=\{A_e\}_{e\in E_{S,R}}$ over a field~$\F$ of the matroid $N_{S,R}$  with respect to the basis $\{b^{(1)},b^{(2)},b^{(3)}\}$ such that Equation~\eqref{eq:rank_2} holds.
	\Cref{pro:weak_representation_G} shows that there exists a group $G_{\A}$ that is a finitely generated subgroup of $GL_c(\F)$ with a group homomorphism $\phi_{\A}: G_{S,R}\rightarrow G_{\A}$.
	By construction of $G_\A$ we can compute using again \Cref{lem:block_matrices}~a)
	\begin{align*}
		r^c_{\A} ( \{w^{(1)},e^{(1)}\})&= \frac{1}{c} \rk\begin{bsmallmatrix}
			-I_c	   & -I_c   \\
			\phi_{\A}(w) & \phi_{\A}(e)  \\
			0 &  0 	  
		\end{bsmallmatrix}\\
		& = 1 + \frac{1}{c} \rk  (\phi_{\A}(w)-I_c).
	\end{align*}
	Thus, Equation \eqref{eq:rank_2} implies $\phi_{\A}(w)\neq I_c$.
	
	By Malcev's thoerem the group $G_{\A}$ is residually finite since it is a finitely generated linear group~\cite{Mal40}.
	Therefore,  there exists a finite group $H$ with a group homomorphism $\phi_H:G_{\A} \rightarrow H$ such that $\phi_H(\phi_{\A}(w))\neq e_H$ where $e_H$ is the neutral element in $H$.
	Hence, the pair $(H,\phi_H \circ \phi_{\A})$ is an instance of $\langle S\mid R \rangle $ with $\phi_H(\phi_{\A}(w))\neq e_H$ which means the given UWPFG instance has a negative answer.
\end{proof}

\section{Algebraic Inflation}\label{sec:alg_reg}

We develop an algebraic inflation procedure as outlined in \Cref{sec:outline} to produce a $c$-admissible arrangement from  a weak $c$-representation.
The inflation consists of two steps.
Both steps use an elementary inflation procedure which we describe first.

\subsection{Elementary Inflation}
Let $\mathcal{U}=\left\{ U_{e}\right\} _{e\in E}$ be a subspace arrangement, $c\in \N$ and $S\subseteq E$ a subset.

Intuitively, the idea is to pick a subspace $W$ of the ambient vector space of the arrangement, and then to extend each subspace $U_e$ by a generic $c$-dimensional subspace of $W$.

Formally, denote the ambient vector space of $\U$ by $V$ and embed $V$ together with the arrangement $\U$ in a larger vector space $\widetilde{V}$ of large enough dimension.
Let $S\subseteq E$ and let $W\subseteq \widetilde{V}$ be a subspace of dimension at least $c$.
Note that $W$ may intersect $V$ non-trivially.
In this situation we construct a new subspace arrangement $\widetilde{\mathcal{U}}$ as follows.
\begin{enumerate}
	\item Choose $|S|$-many generic subspaces $W_{1},\ldots,W_{|S|} $ of $W$, each of dimension $c$.
	\item Denote $S=\{s_1,\dots,s_{|S|}\}$. The new subspace arrangement $\widetilde{\mathcal{U}}$ lives in $\widetilde{V}$ and consists of the subspaces $\widetilde{U}_{s_i} \coloneqq U_{s_i}+W_i$ for $i=1,\dots,|S|$ together with $\widetilde{U}_{e} \coloneqq U_e$ for all $e \in E\setminus S$.
\end{enumerate}

\begin{rmrk} Up to an automorphism of $\widetilde{V}$ fixing $V$, a subspace $W\subseteq \widetilde{V}$ is determined by its dimension together with its intersection with $V$.
	This will suffice for our uses of this construction, and we will give this data instead of constructing $W,\widetilde{V}$ in what follows.
	
	Note in particular that it suffices to take $\tilde{V}$ of dimension $\dim(V) + \dim(W)$.
\end{rmrk}
\begin{defn}
	The arrangement resulting from an application of the elementary inflation construction above to the arrangement $\U$, the subset $S \subseteq E$, and a subspace $W$ of dimension $d$ satisfying $W' = W \cap V$ will be denoted by $\mathcal{EI}_c(\U, S, d, W')$. The subspace $W$, when needed explicitly, will be denoted by $\mathcal{W}_c(\U, S, d, W')$.
\end{defn}

It is not difficult to describe the effect of an elementary inflation on the rank function.
To simplify the exposition we make rather strong assumptions on $d$ and $W'$ which are satisfied in all our applications of the elementary inflation below.
\begin{lemm}\label{lem:inflation}
	Let $\U=\{U_e\}_{e\in E}$ be a subspace arrangement in $V$, $S\subseteq E$ a subset, $W'\subseteq V$ a subspace, and $c,d\in \mathbb{N}$.
	Assume that $c(|S|-1)\le d\le c|S|$ and $\dim W'\le c$.
	Define $\U' \coloneqq \mathcal{EI}_c(\U, S, d, W')$.
	Then for any $T\subseteq E$ it holds
	\[
	\dim U'_T -\dim U_T =\begin{cases}
	c|S\cap T|, & S \not\subseteq T,\\
	d-\dim(U_T \cap W' ), & S\subseteq T.
	\end{cases}
	\] 
\end{lemm}

\begin{proof}
	Denote the subspace used in the elementary inflation by $W\coloneqq \mathcal{W}_c(\U, S, d, W')$ and $W_{S\cap T}\coloneqq \sum_{t\in T\cap S} W_t$ where $W_t$ are the generic subspaces in $W$ as defined in the construction of the elementary inflation.
	Then, we can compute
	\begin{align}
		\dim U'_T -\dim U_T &= \dim(U_T+W_{S\cap T}) - \dim U_T \nonumber\\
		& = \dim(W_{S\cap T}) - \dim(U_T\cap W_{S\cap T}) \nonumber\\
		& = \min\{ d,c|S\cap T| \} - \dim(U_T\cap W_{S\cap T}) ,\label{eq:inflation1}
	\end{align}
	where the last equality holds since the subspaces $W_t$ are chosen generically in the subspace $W$ of dimension $d$.
	Now, we distinguish two cases:
	\begin{description}
		\item[Case 1] $S\not\subseteq T$. By the assumption on $W$ we have $\dim(U_T\cap W)\le \dim (W)\le c$.
		Thus, we have $\dim(U_T\cap W)\le d-c|S\cap T|$ by the assumption on $d$.
		Hence, the genericity of $W_t$  implies $\dim(U_T\cap W_{S\cap T})=0$ and we obtain by \Cref{eq:inflation1}  $\dim U'_T -\dim U_T = c|S\cap T|$.
		\item[Case 2] $S\subseteq T$. In this case, $W_{S\cap T}=W$ which implies that $U_T\cap W_{S\cap T}=U_T\cap W'$.
		Therefore, \Cref{eq:inflation1} implies  $\dim U'_T -\dim U_T = d - \dim(U_T\cap W')$.\qedhere
	\end{description}
\end{proof}

\subsection{Extensions and Full Arrangements}\label{sec:sub_ext}

The main idea of the inflation construction is to extend a given weak $c$-representation outside of the subspace spanned by the basis of the matroid in such a way that, after sufficiently many applications of the procedure, the ranks of the subspaces no longer depend on the particular weak representation but only on the combinatorics of the matroid. Further, the original weak representation can be reconstructed from any iterated inflation (but this last property does not hold for the rank functions).

Before describing the details of the construction, we define a class of subspace arrangements containing those that arise from iterated inflations.

\begin{defn}\label{def:extension}
	Let $\U=\left\{ U_{e}\right\} _{e\in E}$ be a subspace arrangement in a vector space $V$ and let $M=(E,r)$ be a triangle matroid with a distinguished basis $B$.
	We call $\U$ an \emph{extension of a weak $c$-representation of $M$ with respect to $B$}, or for short an extension of $M$, if $\{U_e\cap U_B\}_{e\in E}$ is a weak $c$-representation of $M$ with respect to $B$ and we have for every $T\subseteq E$ and $D\subseteq B$
	\begin{equation}\label{eq:reg}
		\dim(U_T\cap U_D)\le c(r(T)+r(D)-r(T\cup D)).
	\end{equation}
\end{defn}

If inflations of weak $c$-arrangements are to be extensions, it follows that the inflation construction may never modify the dimensions of subspaces corresponding to subsets of the original basis, because one can take $T=D$ in the equation above. 

Further, if $\U=\{U_e\}_{e\in E}$ extends a weak $c$-representation of $M$ and its rank function depends only on the combinatorics of $M$, the dimension of $U_S \cap U_B$ for some $S\subseteq E$ cannot depend on the original weak $c$-representation.
The following definition of a defect gives the difference between the dimension of each intersection $U_S \cap U_B$ and what it ought to be.

\begin{defn}\label{def:regular}
	Let $M=(E,r)$ be a triangle matroid with a distinguished basis $B$ and $\U=\left\{ U_{e}\right\} _{e\in E}$ an extension of a weak $c$-representation of $M$ with respect to $B$.
	For a subset $S\subseteq E$ we define the \emph{defect} of $S$ to be $\de_{\U} (S)= c\cdot r(S)-\dim(U_S\cap U_B)$.
	We call a subset $S\subseteq E$ \emph{full with respect to the basis $B$}, or just full for short,
	if $\de_{\U}(S)=0$.	
\end{defn}

The following lemmas are crucial to our inflation procedure.
At first, we provide bounds for the defect of a subset.

\begin{lemm}\label{lemm:defect}
	Let $M=(E,r)$ be a triangle matroid with distinguished basis $B$ and  $\U=\{U_e\}_{e\in E}$ be an extension of a weak $c$-representation of $M$ with respect to $B$.
	Let $S\subseteq E\setminus B$ such that every $S'\subsetneq S$ is full.
	Then we have
	\[
	0\leq \de_{\U}(S)\le c.
	\]
\end{lemm}
\begin{proof}
	The fact that $\U$ is an extension of $M$ implies $\dim(U_S\cap U_B)\leq c\cdot r(S)$ which implies $0\le \de_{\U} (S)$.
	Now, choose a subset $S'\subsetneq S$ such that $r(S')=r(S)-1$.
	The assumption that $S'$ is full implies
	\[
	c\cdot(r(S)-1)=c\cdot r(S')=\dim(U_{S'}\cap U_B)\le \dim(U_S\cap U_B).
	\]
	Rearranging the terms yields $\de_{\U}(S)\le c$.
\end{proof}

As a second step, we relate the defect with the subset of the basis $C_M(S)$.
Recall, that for a triangle matroid $M=(E,r)$ with basis $B$ the set $C_M(S)\subseteq B$ is defined such that $r(S)=r(S \cup C_M(S))=r(C_M(S))$ if such a subset exists and  $C_M(S)=\emptyset$ otherwise. Since the closures of distinct subsets of a basis are distinct, this defines $C_M(S)$ uniquely.
\begin{lemm}\label{lemm:triangle_regularity}
	Let $M=(E,r)$ be a triangle matroid with distinguished basis $B$ and  $\U=\{U_e\}_{e\in E}$ be an extension of a weak $c$-representation of $M$ with respect to $B$.
	Let $S\subseteq E\setminus B$. If $C_M(S)=\emptyset$ then $S$ is full in $\U$.
\end{lemm}
\begin{proof}
	Set $\A\coloneqq \{A_e\coloneqq U_e\cap U_B\}_{e\in E}$ which is a weak $c$-representation of $M$ by assumption.
	Consider what happens for each value of $r(S)$:
	\begin{description}
		\item[Case 1] $r(S) = 1$. In this case, $c\cdot r(S)=\dim(U_S\cap U_B)$ since $\A$ is a weak $c$-arrangement.
		Hence $S$ is full in $\U$.
		\item[Case 2] $r(S) = 2$. If $S$ is contained in a side of the triangle say $C\subseteq B$, then $C=C_M(S)$ holds and the assumption of the lemma does not hold.

		Suppose $S$ is not contained in a side of the triangle. In that case we have $C_M(S)=\emptyset$.
		Since $S$ contains two elements lying on two different rank two flats of $M$ and the rank of $S$ is also two, we must have $|S|=2$ and we set $S \coloneqq \{s_1,s_2\}$.
		Assume $s_1$ lies on the side $C\subseteq B$ of the triangle, i.e.\ $|C|=2$, and $s_2$ does not lie on the side $C$.
		Hence, by definition of a weak $c$-arrangement we have $\dim A_{C\cup \{s_2\}}=3c$.
		This implies $A_{C\cup \{s_2\}}=A_C \oplus A_{s_2}$.
		Combined with the fact $A_{s_1}\subseteq A_C$ we obtain $\dim A_{S}=2c$.
		Therefore, using the fact that $\U$ is an extension we find
		\[
			2c \ge \dim (U_S\cap U_B) \ge \dim((U_{s_1}\cap U_B)+(U_{s_1}\cap U_B))=\dim A_{S} =2c.
		\]
		Thus, $S$ is full in $\U$.
		\item[Case 3] $r(S) = 3$. In this case $C_M(S)=B$: the closure of $S$ in $M$ is equal to the closure of the entire basis $B$ since the matroid $M$ has rank three.
	\end{description}
	Thus, the lemma holds in each case.
\end{proof}

\begin{rmrk}\label{rmrk:triangle_reg}
	In the notation of the lemma, it follows that if $x,y\in E$ are not contained in any line of the basis $B$ then $\{x,y\}$ is a full subset of $\U$, since in such a situation $r(\{x,y\})=2$.
\end{rmrk}

\begin{lemm}\label{lemm:extension}
	Let $M=(E,r)$ be a triangle matroid with distinguished basis $B$ and  $\U=\{U_e\}_{e\in E}$ be an extension of a weak $c$-representation of $M$ with respect to $B$.
	Let $S\subseteq E$ and let $C_M(S) \subseteq B$ be the subset of the basis defined in \Cref{def:triangle_matroid}.
	If $C_M(S) \neq \emptyset$ we have
	\[
	U_S\cap U_B \subseteq U_{C_M(S)}.
	\]
	In particular, $U_S\cap U_B =U_S\cap  U_{C_M(S)}$.
	
\end{lemm}
\begin{proof}
	Since $\U$ is an extension of $M$ we have
	\[
		\dim (U_{S\cup C_M(S)} \cap U_B) \le c\cdot r(S\cup C_M(S))=c\cdot r(C_M(S))=\dim(U_{C_M(S)}).
	\]
	This implies $U_{C_M(S)} =U_{S\cup C_M(S)}\cap U_B$ since $U_{C_M(S)} \subseteq U_{S\cup C_M(S)}$ and $U_{C_M(S)}\subseteq U_B$ by definition.
	The last inclusion yields
	\[
		U_{C_M(S)} =U_S\cap U_B  +U_{C_M(S)},
	\]
	which directly implies the first claim. The claim $U_S\cap U_B =U_S\cap  U_{C_M(S)}$ follows trivially.
\end{proof}

\subsection{An Inflation Step}
Let $M=(E,r)$ be a triangle matroid and $\U=\left\{ U_{e}\right\} _{e\in E}$ be an extension of a weak $c$-representation of $M$ with respect to $B$.
We now describe an \emph{inflation} procedure that given a subset $S \subseteq E\setminus B$ yields a subspace arrangement $\mathcal{I}(\U,S)$ in which $S$ is full.
In this construction, we assume that any proper subset of $S$ is full in $\U$.
The procedure is split up into two steps.
\begin{description}
	\item[Step 1] We first perform an elementary inflation to inflate the subset $S$.
	We call this step $S$-\emph{inflation}.
	For cases of the form $S=\{x,y\}$, with both $x$ and $y$ lying on the same side of the triangle of $M$, this is depicted in \Cref{fig:regularization_pair}.
	
	We elementary inflate by setting $\U^1 \coloneqq \mathcal{EI}_c(\U, S, c(\vert S \vert - 1) +\de_{\U}(S) , 0)$.
	At the end of this step, we have added a $c$-dimensional subspace to each $U_s$ for $s\in S$.
	Every proper subset of $m$ of these dimension-$c$ subspaces spans a subspace of total dimension $m\cdot c$.
	However, taken all together they span a subspace of dimension $c(\vert S\vert -1) + \de_{U}(S)$, which is in general less than $c\vert S\vert $.
	\item[Step 2] As second step we perform an elementary inflation on the sum of these subspaces with respect to the basis $B$ which we call $B$-\emph{inflation}.
	Again, the case of $S$ equal to two points lying on the same side of the triangle of $M$ is depicted in~\Cref{fig:regularization_basis}.
	
 	While the previous step did not depend on $M$ being a triangle matroid, this step does:
 	Consider the subset $C_M(S)\subseteq B$.
 	\Cref{lemm:triangle_regularity} implies $\de_{\U}(S)=0$ if $C_M(S)=\emptyset$.
	
	Let $W'$ be a generic $\de_{\U}(S)$-dimensional subspace of $U^1_{C_M(S)}$ or $0$ if $C_M(S)=\emptyset$.
	Then we perform an elementary inflation by setting $\U^2\coloneqq \mathcal{EI}_c(\U^1,S, c\vert S \vert ,W')$. 
	
	At the end of this step we have added disjoint $c$-dimensional subspaces to each  $U_s$ for $s\in S$ such that $S$ is a full subset in $\U^2$.
	This will be proved in \Cref{cor:after_reg}.
	We set $\mathcal{I}(\U,S)\coloneqq \U^2$.
	
\end{description}

The next theorem describes the difference of the rank functions after both inflation steps using the results of the previous two subsections.

\begin{theorem}\label{theo:alg_reg}
	Let $\U$ be an extension of a weak $c$-representation of a triangle matroid $M = (E,r)$ with respect to a distinguished basis $B$.
	Let $S \subseteq E \setminus B$ and assume that every subset $S'\subsetneq S$ is full.
	Let $\U' = \mathcal{I}(\U, S)$ be the inflation.
	
	Then if $T\subseteq E$ is any subset disjoint from $S$ and $Z \subseteq S$, we have:
	\[
	r_{\U'}^c(T\cup Z) =
	\begin{cases}
	r^c_\U (T\cup Z) + 2\vert Z \vert, & Z\subsetneq S, \\
	r^c_\U (T\cup S \cup C_M(S)) + 2\vert S\vert - 1, & Z = S.
	\end{cases}
	\]
\end{theorem}
	\begin{proof}
		Using the notation introduced in the definition of the inflation step, we set $\U^1 \coloneqq \mathcal{EI}_c(\U, S, c(\vert S \vert - 1) +\de_{\U}(S) , 0)$ and $\U^2\coloneqq \mathcal{EI}_c(\U^1,S, c\vert S \vert ,W')$ which means $\U^2=\U'=\mathcal{I}(\U, S)$.
		\Cref{lemm:defect} yields $ 0\le \de_{\U}(S)\le c$.
		This implies that both elementary inflations satisfy the assumptions of \Cref{lem:inflation} which we will use repeatedly in the following proof.
		
		Suppose first that $Z\subsetneq S$. \Cref{lem:inflation} yields 
		\[
			\dim(U^2_Z) - \dim(U_Z) = (\dim(U^2_Z) - \dim(U^1_Z)) +( \dim(U^1_Z)- \dim(U_Z) )=2c\vert Z \vert.
		\]
		Similarly, we obtain $\dim(U^2_T + U^2_Z) - \dim(U_T + U_Z) = 2c\vert Z \vert$.
		Thus, the fact $U^2_T=U_T$ implies $\dim(U^2_T \cap U^2_Z)=\dim(U_T \cap U_Z)$.
		Therefore, we can compute
		\begin{align*}
		r_{\U^2}(T\cup Z) & = \dim(U^2_T) + \dim(U^2_Z) - \dim(U^2_T \cap U^2_Z) \\
		& = \dim(U_T) + \dim(U_Z) + 2c|Z| - \dim(U_T \cap U_Z) \\
		& = \dim(U_T) + \dim(U_Z) + 2c|Z| - (\dim(U_T) + \dim(U_Z) - \dim(U_T \cup U_Z)) \\
		& = \dim(U_T \cup U_Z) + 2c|Z|,
		\end{align*}
		so by definition $r^c_{\U^2}(T\cup Z) = r^c_\U(T\cup Z) + 2|Z|$ as required.
		
		Now suppose $Z = S$.
		We first show the following claim
		\begin{claim}\label{claim:sec5}
			$U^2_S \supseteq U^2_{C_M(S)}$.
		\end{claim}
		\begin{proof}[Proof of \Cref{claim:sec5}]
		If $C_M(S)=\emptyset$ this claim is trivial so assume  $C_M(S)\neq\emptyset$.
		Then we have by the construction of the elementary inflation $U^2_S \supseteq U_S+W'$ where $W'$ is subspace of $U_{C_M(S)}$ chosen in the elementary inflation.
		Therefore,
		\begin{equation}\label{eq:W'}
			U^2_S \cap U^2_{C_M(S)} \supseteq (U_S+W') \cap U_{C_M(S)} = (U_S\cap U_{C_M(S)})+W',
		\end{equation}
		where the sum distributes since $W'\subseteq U_{C_M(S)} $.
		 Consider the sum $(U_S\cap U_{C_M(S)})+W'$. By \Cref{lemm:extension}, $U_S \cap U_B = U_S \cap U_{C_M(S)}$.
		 Thus, 
		 \[
		 	\dim(U_S \cap U_{C_M(S)}) = \dim(U_S \cap U_B) = c\cdot r(S) - \de_{\U}(S) = \dim(U_{C_M(S)})  - \de_{\U}(S).
		 \] 
		 Since $W'$ is a $\de_{\U}(S)$-dimensional generic subspace of $U_{C_M(S)}$ we obtain $\dim((U_S\cap U_{C_M(S)})+W') = c\cdot r(C_M(S))$. The summands are each contained in $U_{C_M(S)}$ and the dimension of the sum is $\dim(U_{C_M(S)})$.
		 We obtain $U_{C_M(S)} = (U_S\cap U_{C_M(S)})+W'$.
		 Thus \Cref{eq:W'} yields $U^2_S \supseteq U^2_{C_M(S)}$ in $\U^2$.
		\end{proof}
		Using \Cref{claim:sec5}, we may write 
		\[
			r_{\U^2}^c(T\cup Z) = r_{\U^2}^c((T\cup C_M(S))\cup (S\cup C_M(S))).
		\]
		Therefore, we obtain 
		\begin{equation}\label{eq:ru1}
			r_{\U^2}(T\cup Z) = \dim(U^2_T + U^2_{C_M(S)}) + \dim (U^2_S + U^2_{C_M(S)}) - \dim((U^2_T + U^2_{C_M(S)}) \cap (U^2_S+ U^2_{C_M(S)})).
		\end{equation}
		By construction, we have $ U^2_T + U^2_{C_M(S)} = U_T + U_{C_M(S)}$.
		Using \Cref{lem:inflation} and the exact definition of the elementary inflation, we can compute
		\begin{align}
			\dim (U^2_S + U^2_{C_M(S)}) &=\dim (U^1_S + U^1_{C_M(S)} + c|S|-\de_{\U}(S) )\nonumber\\
			&= \dim (U_S + U_{C_M(S)}) + c(|S|-1) + \de_{\U}(S)+ c|S|-\de_{\U}(S)\nonumber\\
			&= \dim (U_S + U_{C_M(S)}) + c(2|S|-1). \label{eq:S} 
		\end{align}
		By construction, we have $ U^2_T + U^2_{C_M(S)} = U_T + U_{C_M(S)}$.
		Thus, we can compute using \Cref{eq:S} and its analogous form for $U^2_T + U^2_{C_M(S)} + U^2_S+ U^2_{C_M(S)}$
		\begin{align}
		&\dim((U^2_T + U^2_{C_M(S)}) \cap (U^2_S+ U^2_{C_M(S)})) 		\nonumber\\
		= &\dim(U^2_T + U^2_{C_M(S)}) +\dim (U^2_S+ U^2_{C_M(S)}) -\dim(U^2_T + U^2_{C_M(S)} + U^2_S+ U^2_{C_M(S)}) \nonumber\\
		=&\dim(U_T + U_{C_M(S)}) + \dim(U_S + U_{C_M(S)})+ c(2|S|-1) \nonumber\\
		 &- \dim(U_T +U_{C_M(S)}+ U_S + U_{C_M(S)})-c(2|S|-1)\label{eq:ru2} \nonumber\\
		 =&\dim(U_T + U_{C_M(S)}) + \dim(U_S + U_{C_M(S)}) - \dim(U_T + U_S + U_{C_M(S)}).\nonumber
		\end{align}
		This fact, combined with \Cref{eq:ru1,eq:S} implies
		\begin{align*}
		r_{\U^2}(T\cup Z)  =	&  \dim(U_T + U_{C_M(S)}) + \dim(U_S + U_{C_M(S)}) + c(2|S|-1) \\
		&-( \dim(U_T + U_{C_M(S)}) + \dim(U_S + U_{C_M(S)}) - \dim(U_T + U_S + U_{C_M(S)})) \\
		& = \dim(U_T + U_S + U_{C_M(S)}) + c(2|S| - 1).
		\end{align*}
		This implies $r_{\U^2}^c (T\cup Z) = r^c_U(T \cup S \cup {C_M(S)}) + 2|S| - 1$, as claimed.
	\end{proof}

The last theorem enables us to prove that $S$ is full in the inflation $\mathcal{I}(\U, S)$.
\begin{coro}\label{cor:after_reg}
	Let $\U$ be an extension of a weak $c$-representation of a triangle matroid $M = (E,r)$ with respect to a distinguished basis $B$.
	Let $S \subseteq E \setminus B$ be a subset such that every $S'\subsetneq S$ is full and let $\U' = \mathcal{I}(\U, S)$ be the inflation.
	Then $\U'$ is an extension of $M$ and $S$ is full in $\U'$.
\end{coro}
\begin{proof}
	Since $U'_e\cap U'_B=U_e\cap U_B$ for any $e\in E$ by construction of the inflation, the subspace arrangement obtained by intersecting $\U'$ with $U'_B$ is a weak $c$-representation of $M$ since $\U$ is an extension of $M$ by assumption.
	
	Second, let $T\subseteq E$ and $D\subseteq B$.
	We need to show 
	\begin{equation}\label{eq:extension}
		\dim (U'_T\cap U'_D) \le c(r(T)+r(D)-r(T\cup D)).
	\end{equation}
	Suppose $S\not\subseteq T$. 
	Then \Cref{theo:alg_reg} combined with the fact $U'_D=U_D$ and the usual dimension formula implies $\dim (U'_T\cap U'_D)=\dim (U_T\cap U_D)$.
	Therefore, \Cref{eq:extension} holds in this case due to the analogous statement for the extension $\U$.
	
	Now suppose $S\subseteq T$.
	Using as above \Cref{theo:alg_reg} and the extension property on $\U$ yields 
	\begin{align*}
		\dim (U'_T\cap U'_D) &=\dim (U_{T\cup C_M(S)}\cap U_D) \\ 
		& \le c(r(T\cup C_M(S)) + r(D) - r(T\cup C_M(S)\cup D) )\\
		&=c(r(T)+r(D)-r(T\cup D)),
	\end{align*}
	where we used in the last equality $r(S)=r(S\cup C_M(S))$ by the definition of $C_M(S)$.
	
	Lastly, we prove that $S$ is full in $\U'$.
	Using the notation as in the construction of the inflation, \Cref{lem:inflation} implies 
	\[
		\dim(U'_S)= \dim(U_S) + c(|S|-1) +\de_{\U}(S) +c|S|  -\dim (W'\cap U_S).
	\]
	Since $W'$ is a generic subspace of $U_{C_M(S)}$ of dimension $\de_{\U}(S)$ which equals $\dim(U_{C_M(S)}) -\dim (U_S\cap U_B)$ we obtain $\dim (W'\cap U_S)=0$.
	Therefore, we compute as in the proof of \Cref{theo:alg_reg}:
	\begin{align*}
		\dim (U'_S\cap U'_B) & = \dim(U'_S) + \dim(U'_B) - \dim (U'_S+U'_B)\\
		& = \dim(U_S)  +\de_{\U}(S) + \dim(U_B) - \dim (U_S+U_B) \\
		&=\dim (U_S\cap U_B)  +\de_{\U}(S) =c\cdot r(S)\qedhere
	\end{align*}
\end{proof}

\section{Combinatorial Inflation}\label{sec:comb_reg}
This section describes a combinatorial inflation procedure for polymatroids which mirrors the algebraic one described in the previous section.
\begin{defn}\label{def:partial_polymatroid}
	Let $M=(E,r)$ be a rank three matroid with a distinguished basis $B$.
	We call a polymatroid $g$ defined on $E$ an \emph{extension} of $M$ if for all $C\subseteq B$ and $S\subseteq E$ it satisfies
	\begin{equation}\label{eq:partial}
	g(C)+g(S)-g(S\cup C)=r(C)+r(S)-r(S\cup C).\tag{*}
	\end{equation}
\end{defn}
The condition in \Cref{eq:partial} reflects the condition of an subspace arrangement extension given in \Cref{def:extension}.
It ensures that subspace arrangements representing $g$ are weak $c$-representations of $M$ when intersected with the subspace corresponding to $B$
(this statement will be proved in \Cref{thm:intersection_thm}).
Note that for any $C\subseteq B$ applying Equation $\eqref{eq:partial}$ with $S=C$ implies $g(C)=r(C)$.

We define a combinatorial inflation operation on the family of all extension polymatroids $g:\mathcal{P}(E)\rightarrow\R_{\ge 0}$
over a triangle matroid $M=(E,r)$ with a distinguished basis $B$.
This mirrors the algebraic inflation construction - compare \Cref{theo:alg_reg}. Further comparison of these constructions is carried out in the next section.
\begin{defn}\label{def:comb_reg}
	Given $g:\mathcal{P}(E)\rightarrow \R_{\ge 0}$ which is an extension of a triangle matroid $M$ with distinguished basis $B$, together with a subset $S\subseteq E \setminus B$, we define the inflated polymatroid $g'$ as follows:
	let $T \subseteq E$ be any subset disjoint from $S$, and let $Z \subseteq S$.	
	Then we define
	\[
	g'(T\cup Z) \coloneqq \begin{cases}
	g(T\cup Z) + 2|Z|, & Z\subsetneq S,\\
	g(T \cup S\cup C_M(S)) + 2|S| - 1, & Z = S.
	\end{cases}
	\]
The rank function $g'$ resulting from this construction, applied to $g$ and the subset $S$, will be denoted by $\mathcal{I}_\text{comb}(g,S).$
\end{defn}
\begin{prop}\label{prop:comb_reg}
	Let $g$ be a polymatroid extending a matroid $M=(E,r)$ with respect to the distinguished basis $B$ and let $S\subseteq E \setminus B$.
	Then $g' = \mathcal{I}_\text{comb}(g,S)$ also extends $M$ with respect to $B$.
\end{prop}

\begin{proof}
	Let $T\subseteq E$ be any subset disjoint from $S$. 
	Since $g$ is a polymatroid extending $M$, we obtain by \Cref{eq:partial} that
	\[
	g(S) + g(C_M(S)) - g(S\cup C_M(S)) = r(S) + r(C_M(S)) - r(S\cup C_M(S)).
	\]
	Together with $g(C_M(S)) = r(C_M(S))$ and $r(S)=r(S\cup C_M(S))$, this equation implies $g(S) = g(S\cup C_M(S))$.
	The fact that $g$ is a polymatroid yields $g(T \cup S \cup C_M(S)) = g(T \cup S)$.
	Thus, we may rewrite the definition of $g'$ by
	
	\[g'(T\cup Z) \coloneqq \begin{cases}
	g(T\cup Z) + 2|Z|, & Z\subsetneq S,\\
	g(T \cup Z) + 2|S| - 1, & Z = S.
	\end{cases}
	\]
	
	This is a polymatroid, since it is the sum of $g$ with the rank function of a matroid on $S$, namely the sum of a free rank function with a uniform rank function of rank $|S|-1$.
	
	It remains to demonstrate that $g'$ also satisfies \Cref{eq:partial}.
	Let $T'\subseteq E$ and $D\subseteq B$.
	The following equalities hold by definition of $g'$:
	\begin{align*}
	g'(T') &=  \begin{cases} g(T') + 2|T'\cap S|, &  S\not \subseteq T', \\
															 g(T') + 2|T'\cap S| - 1, & S\subseteq T' ,
														\end{cases}  \\
	g'(T'\cup D) &= 						  \begin{cases} g(T' \cup D) + 2|T'\cap S|, &  S\not \subseteq T',\\
																	g(T'\cup D) + 2|T'\cap S| - 1, & S\subseteq T' .
														\end{cases}
	\end{align*}
	Thus, it is always true that $g'(T') - g'(T'\cup D) = g(T') - g(T'\cup D)$.
	Using $g'(D)=g(D)$, we obtain in total
	\[
	g'(D)  + g'(T')  - g'(T'\cup D) = g(D) +g(T') - g(T'\cup D). \]
	We conclude $g'$ satisfies \Cref{eq:partial}, since $g$ does.
\end{proof}

\section{Compatibility of Algebraic and Combinatorial Inflation}\label{sec:comp}
This section proves two theorems which relate the algebraic and combinatorial inflation procedures introduced in the last two sections.
We start by giving a theorem that establishes a connection between weak $c$-representations and combinatorial polymatroid extensions.

\begin{theorem}\label{thm:intersection_thm}
	Let $M=(E,r)$ be a matroid with a distinguished basis $B$ and let $g:\mathcal{P}(E)\rightarrow\R_{\ge 0}$ be a polymatroid extending $M$
	(i.e. $g$ satisfies Equation~\eqref{eq:partial} for all appropriate subsets).
	Suppose $\U=\left\{ U_{e}\right\} _{e\in E}$
	represents $c\cdot g$, that is $r_{\U}=c\cdot g$.
	
	Denote $ A_{e}= U_{e}\cap U_B$ for each $e \in E$. Then the arrangement $\mathcal{A}=\{ A_{e}\} _{e\in E}$
	is a weak $c$-representation of $M$ with respect to $B$.
\end{theorem}
\begin{rmrk}
	The proof is just an application of \eqref{eq:partial} and basic linear algebra. The fact that it works is what justifies the definition of polymatroid extensions.
\end{rmrk}
\begin{proof}
	By definition, for any $C\subseteq B$ we have $g(C)=r(C)$ and $U_{C}=A_{C}$.
	Thus ${r_{\A}^{c}{\big|_B}	=r{\big|_B}	}$. 
	
	Let $e\in E\setminus B$ and $C\subseteq B$.
	Then, denoting $S\coloneqq C\cup\left\{ e\right\} $ we obtain
	\begin{align}
	\nonumber \dim A_{S}=&\dim A_{C}+\dim A_{e}-\dim(A_{C}\cap A_{e})\\
	&\dim A_{C}+\dim(U_{e}\cap U_{B})-\dim(U_{C}\cap(U_{e}\cap U_{B})).\label{eq:intersection}
	\end{align}
	Note that $U_{C}\cap U_{e}\cap U_{B}=U_{C}\cap U_{e}$ since $C\subseteq B$.
	Using \Cref{eq:intersection} together with  the dimension formula and the identities $\dim U_{T}=c\cdot r_{\mathcal{U}}^{c}(T)=c\cdot g(T)$
	for any subset $T\subseteq E$, we obtain as required
	\begin{align*}
	\frac{1}{c}\dim A_{S}&=r(C)+\left[g(B)+g(\{e\})-g(B\cup \{e\})\right]-\left[g(C)+g(\{e\})-g(C\cup \{e\})\right]\\
	&\stackrel{\eqref{eq:partial}}{=}r(C)+ \left[r(B)+r(\{e\})-r(B\cup \{e\})\right] -\left[ r(C)+r(\{e\})-r(C\cup \{e\})\right]\\
	&=r(C\cup \{e\}).
	\end{align*}
	
	Suppose now that $S\subseteq E$ is a general subset. Then 
	\begin{align*}
	r_{\A}(S) &=\dim(A_S) =\dim\left(\sum_{x\in S}(U_{x}\cap U_{B})\right)\\
	& \le\dim\left(U_S\cap U_{B}\right)\\
	&=\dim U_{S}+\dim U_{B}-\dim U_{S\cup B}\\
	&= c(g(S)+g(B)-g(S\cup B))\\
	&\stackrel{\eqref{eq:partial}}{=}c(r(S)+r(B)-r(S\cup B))=c\cdot r(S).
	\end{align*}
	Therefore, $\A$ is a weak $c$-representation of $M$ as claimed.
\end{proof}
\begin{rmrk}
	In Equation~\eqref{eq:intersection} we needed the fact that $e$ is a single element of	the ground set:
	For a general $S\subset E$ the subspace $A_{S}=\sum_{e\in S}(U_{e}\cap U_{B})$
	may not equal $U_{S}\cap U_{B}$.
\end{rmrk}

\begin{prop}\label{prop:contraction_formula}
	Let $M=(E,r)$ be a triangle matroid with distinguished basis $B$. Let $D\subseteq B$, $S\subseteq E$.
	Let $C_M(S)\subseteq B$ as defined in \Cref{def:triangle_matroid}.
	Any polymatroid $g:\mathcal{P}(E)\rightarrow \mathbb{R}_{\ge 0}$ extending $M$ satisfies 
	\[g(D\cup S\cup C_M(S)) = g(S\cup B) - g(B) + r(D\cup S).\]
	Similarly, any subspace arrangement $\U=\{U_e\}_{e\in E}$ extending a weak $c$-representation of $M$ satisfies 
	\[r^c_{\U}(D\cup S\cup C_M(S)) = r^c_{\U}(S\cup B) - r^c_{\U}(B) + r(D\cup S).\]
\end{prop}
\begin{rmrk}
	The point of this is the following: let $r^c_{\U}$ be the rank function of an arrangement as above, and suppose it has been obtained from a weak $c$-representation $\A$ of $M$ by a sequence of inflations. Let $r^c_{\U'}$ be the function obtained by inflating at an additional subset $S$. Then $r^c_{\U'}(D\cup S)$ can be expressed in terms of $r^c_{\U}(D\cup S\cup C_M(S))$ by~\Cref{theo:alg_reg}.
	
	The formula above thus expresses $r^c_{\U'}(D\cup S)$ in terms of two simpler objects, namely $r$ and the rank function
	\[S\mapsto r^c_{\U}(S\cup B) - r^c_{\U}(B)\] 
	obtained by contracting $B$ (or quotienting out $U_B$). 
	
	This quotient is easier to handle than $r^c_{\U}$: its rank function does not depend on $\A$, but only on $M$ and the sequence of inflations we applied. We will see this explicitly later in this section.
\end{rmrk}
\begin{proof}[Proof of \Cref{prop:contraction_formula}]
	By \Cref{eq:partial} applied to $D\cup S\cup C_M(S)$ and $B$ we have 
	\[
		g(D\cup S\cup C_M(S)) + g(B) - g(S\cup B) = r(D\cup S\cup C_M(S)) + r(B) - r(S\cup B).
	\]
	Since $B$ is a basis for $M$, we have $r(B) = r(S\cup B)$ and the last two terms on the right cancel.
	Similarly, by definition of $C_M(S)$ we have $r(D\cup S\cup C_M(S)) = r(D\cup S)$.
	By the discussion after \Cref{def:partial_polymatroid} we see $g(B) = r(B)$.
	Substituting, we obtain \[g(D\cup S\cup C_M(S)) + g(B) - g(S\cup B) = r(D\cup S),\] and rearranging yields the required equation on $g$.
	
	For the claim on $\U$ we will work with dimensions rather than the function $r_\U$, since the latter provides no mechanism for considering intersections of subspaces. We have
	\[\dim(U_{D\cup S\cup C_M(S)}) + \dim(U_B) = \dim(U_{D\cup S\cup C_M(S)} + U_B) + \dim(U_{D\cup S\cup C_M(S)} \cap U_B).\]
	Note $U_{D\cup S\cup C_M(S)}+U_B = (U_D + U_S + U_{C_M(S)}) + U_B = U_S + U_B$, since $D,C_M(S)\subseteq B$. 
	Further, $U_{D\cup S\cup C_M(S)} \cap U_B = (U_{D\cup C_M(S)} + U_S) \cap U_B$. Since $U_{D\cup C_M(S)} \subseteq U_B$, the intersection distributes over the sum.
	This yields \[(U_{D\cup C_M(S)} + U_S) \cap U_B = (U_{D\cup C_M(S)} \cap U_B) + (U_S \cap U_B). \] 
	There are now two cases to consider:
	\begin{description}
		\item[Case 1] $C_M(S) \neq \emptyset$. Since $\U$ extends a weak $c$-representation of $M$ \Cref{lemm:extension} implies
		\[U_S \cap U_B \subseteq U_{C_M(S)} \subseteq U_{D\cup C_M(S)},\] and $U_{D\cup C_M(S)} = U_{D\cup C} \cap U_B$. Thus
		\[(U_{D\cup C_M(S)} \cap U_B) + (U_S \cap U_B) = U_{D\cup C_M(S)}.\] 
		Further, $\dim(U_{D\cup C_M(S)}) = c\cdot r(D\cup C_M(S)) = c\cdot r(D\cup S)$, where the first equality holds since $\U$ extends a weak $c$-representation of $M$.
		\item[Case 2] $C_M(S)=\emptyset$. In this case, $S$ is full in $\U$ by \Cref{lemm:triangle_regularity}.
		Thus, we obtain $\dim(U_S \cap U_B) = c\cdot r(S)$. Since $D\cup C_M(S) = D \subseteq B$, this implies:
		\begin{align*}
		\dim((U_{D\cup C_M(S)} \cap U_B) + (U_S \cap U_B)) &= \dim(U_D + (U_S \cap U_B)) \\
		&= \dim(U_D)+ \dim(U_S \cap U_B) - \dim(U_D \cap U_S \cap U_B) \\
		&= c\cdot(r(D) + r(S) - [r(D) + r(S) - r(D\cup S)]) \\
		&= r(D\cup S).
		\end{align*}
	\end{description}
	Thus in either case we obtain $\dim((U_{D\cup C_M(S)} + U_S) \cap U_B) = c\cdot r(D\cup S),$ and on substituting this into the previous equation:
	\[\dim(U_{D\cup S\cup C_M(S)}) + \dim(U_B) = \dim(U_{S\cup B}) + c\cdot r(D\cup S).\]
	Rearranging and replacing the dimensions with ranks gives the claim.
\end{proof}

\begin{theorem}\label{thm:regularization}
	Let $M=(E,r)$ be a triangle matroid with distinguished basis $B$. Then there is a polymatroid $g$ extending $M$ such that $M$ has a  weak $c$-representation with respect to $B$ if and only if $c\cdot g$ has a subspace arrangement representation $\U$, that is $r_{\U}^c=g$.
	
	Further, given a weak $c$-representation $\A$ of $M$, the subspace arrangement $\U$ representing $c\cdot g$ can be chosen to extend $\A$.
\end{theorem}
\begin{proof}
	Choose a linear ordering $S_0 = \emptyset, S_1,\dots,S_n$ on $\mathcal{P}(E\setminus B)$ which refines the ordering given by inclusion.
	That is, if $S_i \subseteq S_j$ then $i \le j$.
	We inductively define a sequence of polymatroids $g_0 , \ldots,  g_n$: set $g_0\coloneqq r$ and given $g_i$, define $g_{i+1} \coloneqq \mathcal{I}_\text{comb}(g_i,S_{i+1})$ for $i=0,\dots,n-1$.
	Finally, set $g\coloneqq g_n$.
	\Cref{prop:comb_reg} implies that $g$ is a polymatroid extension of $M$ with respect to $B$.
	
	\Cref{thm:intersection_thm} implies that if $c\cdot g$ is representable as the rank function of a subspace arrangement then $M$ is weakly $c$-representable.
	
	For the other implication, suppose a weak $c$-representation $\mathcal{A}=\{A_e\}_{e\in E}$ of $M$ is given.
	We inductively produce a sequence of subspace arrangements $\U_0  \ldots, \U_n$ with $\mathcal{U}_i=\{U_{i,e}\}_{e\in E}$ as follows: set $\U_0\coloneqq \A$ and given $\U_i$, define $\U_{i+1} \coloneqq \mathcal{I}(\U_i, S_{i+1})$ for $i=0,\dots,n-1$. Lastly, set $\U\coloneqq \U_{n}$.
	
	Note we never inflate with respect to $S_0 = \emptyset$.
	Furthermore, the choice of order on $\mathcal{P}(E\setminus B)$ and \Cref{cor:after_reg} implies that in each step every proper subset of $S_i$ is full in $\U_{i-1}$.
	Thus the assumptions of \Cref{theo:alg_reg} are satisfied in each inflation step.
	
	We perform an induction which, comparing the combinatorial inflation operation $\mathcal{I}_\text{comb}$ with the linear-algebraic inflation operation $\mathcal{I}$, proves $\mathcal{U}$ represents $c\cdot g$ as desired.
	
	We do this using the following two claims.
	
	\begin{claim}\label{claim:1}For any $1\le i\le n$ and any $T\subseteq E$: $c\cdot g_i (B\cup T) = r_{\U_i}(B\cup T).$
	\end{claim}
	
	Note that the contractions of $c\cdot g_i$ and $r_{\U_i}$ by $B$ are given by \[T\mapsto c\cdot g_i (B\cup T) - c\cdot g_i(B) \quad\mathrm{and}\quad T\mapsto r_{\U_i}(B\cup T) - r_{\U_i}(B)\]respectively. Since $c\cdot g_i(B) =r_{\U_i}(B)$, the claim shows that these contractions are equal.
	
	Once this has been shown, we will use it to prove a slightly more general statement.
	
	\begin{claim}\label{claim:2} If $D\subseteq B$ and $j\le i$ then $c\cdot g_i (D\cup S_j) = r_{\U_i}(D\cup S_j).$
	\end{claim}
	
	The theorem itself is directly implied by \Cref{claim:2} for the case $i=n$, since any subset of $E$ may be written as $A \cup S_j$ with $A\subseteq B, 1\le j\le n$.
	
	\begin{proof}[Proof of \Cref{claim:1}]
		For $i=0$ and for any $T\subseteq E$, \[g_i(T\cup B) = r^c_{\U_i}(T\cup B) = r(B)\] by definition.
		
		Let $i>0$. Applying \Cref{def:comb_reg} \footnote{In the notation of \Cref{def:comb_reg}, we take $A = (T\cup B) - S_i$ and $Z = (T\cup B)\cap S_i = T\cap S_i$.} we see that for $C_M(S_i) \subseteq B$ as in \Cref{def:triangle_matroid}
		\[
		g_i(T\cup B)=\begin{cases}
		g_{i-1}(T\cup B)+2|T\cap S_i|, & S_i\not\subseteq T,\\
		g_{i-1}(T\cup B\cup C_M(S_i))+2|T\cap S_i|-1, & S_i\subseteq T.
		\end{cases}\] 
		Thus, $g_{i}\left(T\cup B\cup C_M(S_i)\right)=g_{i}\left(T\cup B\right)$.
		Subtracting $g_{i-1}\left(T\cup B\right)$ from both sides, we find
		\[
		g_{i}(T\cup B) - g_{i-1}(T\cup B) = 2|T\cap S_i| - 
		\begin{cases}
		0, & S_i\not\subseteq T, \\
		1, & S_i\subseteq T.
		\end{cases}\]
		
		Applying the same reasoning using \Cref{theo:alg_reg} yields the same formula for $r_{\U_i}(T\cup B) - r_{\U_{i-1}}(T\cup B)$.
		Therefore:
		\begin{align*}
		g_{i}(T\cup B)-r(T\cup B) &= g_{i}(T\cup B)-g_0(T\cup B) \\
		&= \sum_{j=1}^{i}(g_j(T\cup B)-g_{j-1}(T\cup B)) \\
		&= \sum_{j=1}^{i} \left( 2|T\cap S_{i+1}| - 
		\begin{cases}
		0, & S_{i}\not\subseteq T, \\
		1, & S_{i}\subseteq T,
		\end{cases}\right) \\
		&= \sum_{j=1}^{i}(r_{\U_{j}}^c(T\cup B)-r_{\U_{j-1}}^c(T\cup B)) \\
		&= r_{\U_{i}}^c(T\cup B) - r_{\U_0}^c(T\cup B) \\
		&= r_{\U_{i}}^c(T\cup B) - r(T\cup B).
		\end{align*}
		Hence, we have the equality $g_i(T\cup B) = r^c_{\U_i}(T\cup B).$
	\end{proof}
	\begin{proof}[Proof of \Cref{claim:2}]
		We proceed again by induction on $i$.
		The claim is trivially true for $i=0$ by the definition of a weak $c$-arrangement: what it means is that $r(D) = r_\A^c (D)$ for any $D\subseteq B$.
		
		Suppose the claim is true for some $i<n$. Let us show it also holds for $i+1$:
		For $j\le i$ and any $D\subseteq B$ we obtain by setting $Z \coloneqq S_{i+1}\cap S_j$ and $T\coloneqq S_j \setminus Z$:
		\begin{align*}
		g_{i+1}(D\cup S_j) &= g_{i+1}((D\cup T) \cup Z) \\
		&= g_i(D\cup T\cup Z) + 2|Z| \\
		&= g_i(D\cup S_j) + 2|Z|,
		\end{align*}
		and $r^c_{\U_{i+1}}(D\cup S_j) = r^c_{\U_i}(D\cup S_j) + 2|Z|$ in the same way. Applying the claim with the index~$i$, we see \[c\cdot g_{i+1}(D\cup S_j) = c\cdot g_{i}(D\cup S_j) + 2c|Z| = r_{\U_{i}}(D\cup S_j) + 2c|Z| = r_{\U_{i+1}}(D\cup S_j).\]
		
		For $j=i+1$, consider $C_M(S_{i+1})\subseteq B$ as defined in \Cref{def:triangle_matroid}.
		We have by \Cref{theo:alg_reg,prop:contraction_formula}:
		\begin{align*}
		r^c_{\U_{i+1}}(D\cup S_{i+1}) &= r^c_{\U_i}(D\cup C_M(S_{i+1}) \cup S_{i+1}) + 2|S_{i+1}|-1 = \\
		&= r^c_{\U_i}(S_{i+1}\cup B) - r^c_{\U_i}(B) + r(D \cup S_{i+1}) +2|S_{i+1}|-1 \\
		&= r^c_{\U_i}(S_{i+1}\cup B) - r(B) + r(D \cup S_{i+1}) +2|S_{i+1}|-1,
		\end{align*}
		where the last equality holds because $\U_{i}$ extends a weak $c$-representation of $M$, and thus $r^c_{\U_i}(Z) = r(Z)$ for any $Z\subseteq B$.
		
		In exactly the same way, using \Cref{def:comb_reg} in place of \Cref{theo:alg_reg}, we find:
		\[g_{i+1}(D\cup S_{i+1}) = g_i(S_{i+1}\cup B) - r(B) + r(D \cup S_{i+1}) +2|S_{i+1}|-1.\]
		Hence \[r^c_{\U_{i+1}}(D\cup S_{i+1}) - g_{i+1}(D\cup S_{i+1}) = r^c_{\U_i}(S_{i+1}\cup B) - g_i(S_{i+1}\cup B),\]
		and by \Cref{claim:1} the difference on the right side is $0$.
	\end{proof}
	As we noted above, this proves the theorem.	
\end{proof}

\section{Bases of $c$-Admissible Arrangements}\label{sec:c-bases}

This section has two main purposes. The first is to translate questions about polymatroids and $c$-admissible arrangements to questions about matroids and $c$-arrangements. This is carried out in \cref{sec:expansions}.

The second is more directly related to inflations and generalized Dowling geometries: \cref{thm:regularization} gives a method by which to extend a weak $c$-representation $\A$ of a triangle matroid into a $c$-admissible subspace arrangement $\U$. This construction gives an arrangement of a combinatorial type that does not depend on $\A$.

We want to apply group-theoretic undecidability results to this construction, where the weak arrangement $\A$ is constructed from a group presentation as in \cref{sec:staudt}. For this, we need to check whether some two subspaces $A_x,A_y$ of $\A$ are different, and this needs to be encoded in the combinatorics of $\U$'s rank function; but the rank function of $\U$ contains no such information. It does not even know whether $\A$ was constructed from a trivial representation of the group or a faithful one. Thus our second goal is to modify $\U$ in such a way that the resulting rank function contains the required information. We call this construction "forcing an inequality". It is carried out in \cref{sec:inequality}; \cref{sec:well_sep} contains a technical preliminary proposition.

As with previous parts of our proof, an algebraic construction has a parallel combinatorial one. The necessary combinatorial construction will be an easy application of facts from \cref{sec:expansions}.

\subsection{Expansions and $c$-Bases}\label{sec:expansions}
We wish to translate problems about $c$-admissible subspace arrangements and polymatroids to problems about $c$-arrangements and matroids. This is entirely analogous to translating problems on subspace arrangement to problems on vector arrangements. 

Given a subspace arrangement representing a polymatroid $g$, one can construct a vector arrangement from it as follows: pick a basis for every subspace, and take the collection of all resulting basis vectors. If we keep track of which vector came from which subspace, the original subspace arrangement can be reconstructed.

This construction does not depend only on $g$: the result depends on the specific choice of bases. However, if the ground field is large enough and the bases are chosen generically, we always obtain the same matroid. This matroid is called the free expansion $\mathscr{F}(g)$.

In fact there are only finitely many possible vector arrangements obtained by picking bases for a subspace arrangement representing the polymatroid $g$. The matroids arising in this way are called the expansions of $g$, and form a subset of the weak images of $\mathscr{F}(g)$. For further details, cf.\ \cite[Proposition 10.2.6]{Ngu86} or \cite[Chapter 11]{Oxl11} (with slightly different terminology and notation).

We will have use for both the free expansion and for the set of all expansions of $g$. This section collects the relevant definitions and lemmas, where subspace and vector arrangements are systematically replaced by $c$-admissible arrangements and $c$-arrangements respectively.

\begin{defn}[Arrangement $c$-bases and polymatroid expansions]\label{def:expansion}
$ $ 

	\begin{enumerate}
	\item 	Let $\mathcal{U}=\left\{ U_{e}\right\} _{e\in E}$ be a $c$-admissible
	subspace arrangement, and denote $d_{e}=\frac{1}{c}\dim U_{e}$ for
	each $e\in E$. A \emph{$c$-basis} of $\mathcal{U}$ is a $c$-arrangement
	\[\mathcal{W} = \{W_{e,i}\}_{\substack{e\in E, \\ 1\le i\le d_e}}\]
	in the same ambient vector space as $\U$, satisfying that for each $e\in E$:
	\[
	U_{e}=\sum_{i=1}^{d_{e}}W_{e,i}.
	\]
		
	In this situation, we denote $\mathcal{W}_e = \{W_{e,i}\}_{1\le i\le d_e}$.
	
	Note that if $c=1$, the subspaces $W_{e,i}$ are lines, and, after
	identifying each $W_{e,i}$ with an arbitrary nonzero point on it,
	we find that each $\mathcal{W}_{e}$ is a basis of the subspace $U_{e}$.
	This is the sense in which these objects are $c$-bases.
	
	\item	The \emph{combinatorial type of the $c$-basis $\mathcal{W}$}
			is the matroid given by $r^c_\mathcal{W}$.
	
	\item An \emph{expansion} of a polymatroid $(E,g)$ is a matroid with rank function $r$ on the ground set $\{(e,i) \mid e\in E, 1\le i\le g(e)\}$ satisfying the following property for any $S \subseteq E$: \[
	r\left(\{(e,i)\mid e\in S, 1\le i\le g(e)\}\right) = g(S).
	\]
	
	In particular, the combinatorial type of a $c$-basis of $\U$ is an expansion of $r^c_\U$.
	However, the converse does not always hold.
\end{enumerate}
\end{defn}
\begin{defn}
Let $\mathcal{U}=\left\{ U_{e}\right\} _{e\in E}$ a $c$-admissible subspace arrangement.
Denote $d_{e}=\frac{1}{c}\dim\left(U_{e}\right)$ for each $e\in E$.

In each $U_{e}$, choose $d_{e}$ generic subspaces $W_{e,1},\dots,W_{e,d_E}$, each of dimension $c$.
The arrangement $\{W_{e,i}\}_{1\le i\le d_e}$ is called a \emph{generic $c$-basis} of $\mathcal{U}$. 
\end{defn}

It is a consequence of the following lemma that a generic $c$-basis is a $c$-arrangement.

\begin{lemm}[Splitting Lemma]\label{lemm:splitting}Let $\left\{ U_{e}\right\} _{e\in E}$ be a
	$c$-admissible subspace arrangement in a vector space $V$. For each
	$e\in E$ let $k_{e}$ be a nonnegative integer and let $W_{e,1},\ldots,W_{e,k_{e}}$
	be generic subspaces of $U_{e}$, each of dimension $c$. Then 
	$
	\left\{ W_{e,i}\right\} _{\substack{e\in E,\\
			1\le i\le k_{e}
		}
	}
	$
	is a $c$-admissible subspace arrangement.
\end{lemm}
\begin{proof}
	We will prove by induction on $k\coloneqq\sum_{e\in E}k_{e}$ that
	\[
	\mathcal{W}\coloneqq\left\{ W_{e,i}\right\} _{\substack{e\in E,\\
			1\le i\le k_{e}
		}
	}\cup\left\{ U_{e}\right\} _{e\in E}
	\]
	is $c$-admissible. Since a subset of a $c$-admissible arrangement
	is also $c$-admissible, this proves the lemma.
	
	For $k=0$, the statement is just the assumption that $\left\{ U_{e}\right\} _{e\in E}$
	is $c$-admissible.
	
	Suppose the statement is true for $k-1\ge0$, and let $\mathcal{W}$
	be any arrangement as above with $\sum_{e\in E}k_{e}=k-1$. Let $e\in E$
	and let $W\subseteq U_{e}$ be a generic subspace of dimension $c$. We
	wish to show $\mathcal{W}\cup\left\{ W\right\} $ is also $c$-admissible.
	Given $\mathcal{S}\subseteq\mathcal{W}$, denote $W_{\mathcal{S}}=\sum_{U\in \mathcal{S}}U$. We
	want to show
	$
	\dim\left(W_{\mathcal{S}}+W\right)\in c\cdot\mathbb{N},
	$
	and it suffices to prove that 
	\[
	\dim\left(W_{\mathcal{S}}\cap W\right)=\dim\left(W_{\mathcal{S}}\right)+\dim\left(W\right)-\dim\left(W_{\mathcal{S}}+W\right)\in c\cdot\mathbb{N}.
	\]
	Clearly if $U_{e}\subseteq W_{\mathcal{S}}$ we are done, since then also $W\subseteq W_{\mathcal{S}}$
	and $W_{\mathcal{S}}\cap W=W$.
	If $U_{e}\not\subseteq W_{\mathcal{S}}$, consider the intersection
	$U_{e}\cap W_{\mathcal{S}}$: it has dimension $m\cdot c$ for some nonnegative
	integer $m$, with $mc<\dim\left(U_{e}\right)$.
	Since $\dim\left(U_{e}\right)\in c\cdot\mathbb{N}$, we have
	$\dim\left(U_{e}\right)\ge\left(m+1\right)c$.
	By genericity of $W$
	in $U_{e}$, the intersection $W\cap\left(U_{e}\cap W_{\mathcal{S}}\right)$
	is trivial. Thus we have
	\[
	0=W\cap\left(U_{e}\cap W_{\mathcal{S}}\right)=\left(W\cap U_{e}\right)\cap W_{\mathcal{S}}=W\cap W_{\mathcal{S}}.
	\]
	Since $\mathcal{S}\subseteq\mathcal{W}$ was arbitrary, the arrangement $\mathcal{W}\cup\left\{ W\right\} $
	is $c$-admissible as required.
\end{proof}

We now give a description of the matroid underlying a generic $c$-basis. This description is not new (cf. \cite[Proposition 10.2.7]{Ngu86}), except possibly in the context of $c$-arrangements.
We include it for the reader's convenience.

\begin{prop}
	Let $\U=\{U_e\}_{e\in E}$ be a $c$-admissible subspace arrangement. Denote $d_e \coloneqq \frac{1}{c}\dim(U_e)$ for each $e\in E$, and let $\mathcal{W}\coloneqq \{W_{e,i}\}_{\substack{e\in E, \\ 1\le i\le d_e}}$ be a generic $c$-basis.
	 For each $e\in E$, choose some $\mathcal{S}_e \subseteq \{W_{e,i}\}_{1\le i\le d_e}$.
	Denote $\mathcal{S}=\bigcup_{e\in E}\mathcal{S}_e$. Then:
	\[
	\dim\left(\sum_{W\in S}W\right)<c\left|S\right|
	\]
	if and only if there is a subset $F \subseteq E$ such
	that
	\begin{equation}\label{eq:F}
	\dim\left(\sum_{f\in F}U_f\right)<c\sum_{f\in F}\left|\mathcal{S}_f\right|.
	\end{equation}
\end{prop}
\begin{proof}
	The implication ``$\Rightarrow$'' holds by genericity of the subspaces $U_{i,j}$.
	
	For the other direction, assume there is a subset $F \subseteq E$ satisfying Equation~\eqref{eq:F}.
	Then since $\sum_{W\in \mathcal{S}_f}W \subseteq U_f$ for each $f\in F$, we have
	\[
	\sum_{f\in F}\sum_{W\in \mathcal{S}_f}W \subseteq \sum_{f\in F} U_f.
	\]
	Hence \[ \dim\left(\sum_{f\in F}\sum_{W\in S_f}W\right)\le\dim\left(\sum_{f\in F} U_f\right)<c\cdot \sum_{f\in F} |\mathcal{S}_f|.\] Therefore, we obtain:
	\begin{align*}
	\dim\left(\sum_{e\in E} \sum_{W\in \mathcal{S}_e}W\right) &= \dim\left(\sum_{f\in F} \sum_{W\in \mathcal{S}_f}W + \sum_{e\notin F}\sum_{W\in \mathcal{S}_e}W\right) \\
	&\le \dim\left(\sum_{f\in F} \sum_{W\in \mathcal{S}_f}W\right) + \dim\left(\sum_{e\notin F}\sum_{W\in \mathcal{S}_e}W\right) \\
	&< c\sum_{f\in F}|\mathcal{S}_f| + c\sum_{e\notin F}|\mathcal{S}_e| \\
	&= c\sum_{e\in E} |\mathcal{S}_e| = c|\mathcal{S}|. \qedhere
	\end{align*}
\end{proof}
This gives a description
of the independent sets in the matroid given by any generic $c$-basis of $\U$. Following \cite{Ngu86}, we call this matroid the \emph{free expansion} of $r^c_\U$ and denote it by $\mathscr{F}(r^c_{\U})$. Similarly, for a polymatroid $g$ we will use the notation $\mathscr{F}(g)$.

The following lemmas capture the correspondence between $c$-bases
and the combinatorics of polymatroids.
\begin{lemm}\label{lemm:c_basis_expansion1}
	Let $g$ be a polymatroid on $E$ and let $(\{(e,i)\mid e\in E, 1\le i\le g(e)\}, r)$ be an expansion. If \[\mathcal{W} = \{W_{e,i}\}_{\substack{e\in E, \\ 1\le i\le g(e)}}\] is a $c$-arrangement representing $r$, then the arrangement $\U=\{U_e\}_{e\in E}$ where $U_e = \sum_{i=1}^{g(e)} W_{e,i}$ is a subspace arrangement representing $c\cdot g$.
\end{lemm}

This is a direct consequence of the definitions.

\begin{lemm}\label{lemm:c_basis_expansion2}
	Let $g$ be a polymatroid on $E$. Any expansion $M$ of $g$ is a weak image of the free expansion $\mathscr{F}(g)$, that is $M$ and $\mathscr{F}(g)$ are matroids on the same ground set and any independent set in $M$ is also independent in $\mathscr{F}(g)$.
\end{lemm}

For the proof see \cite[Proposition 10.2.6]{Ngu86}.
It follows that $g$ has a finite set of expansions, computable from $g$.
\subsection{Well-Separated Extensions}\label{sec:well_sep}

Let $\A = \{A_e\}_{e\in E}$ be a weak $c$-representation of a generalized Dowling geometry $N_{S,R}$ as constructed in \cref{sec:staudt}.
As remarked at the beginning of this section, if $\U$ is an arrangement resulting from an inflation of $\A$ such that all subsets of $\U$ are full, the rank function $r^c_{\U}$ does not contain enough information to determine whether $A_x \neq A_y$ for certain $x,y\in E$.
However, inequalities of this form, or equivalently $r^c_{\A}(\{x,y\}) > 1$, are precisely what we need in order to apply Slobodskoi's undecidability theorem, using \cref{theo:weak_undecidable}.

To overcome this difficulty, our strategy is to modify $\U$ by adding a subspace $W$ which is contained in $A_x$ but not in $A_y$ (provided that they are in fact distinct). 

The entire proof hinges on the fact that the final rank function $r^c_{\U}$ does not depend on anything other than combinatorial data coming from a UWPFG instance. In particular, it must be independent of $c$. Therefore we need to make sure $W$ can be chosen to have some pre-determined dimension, and for this it is necessary to bound $\frac{1}{c}\dim(A_x \cap A_y)$ away from~$1$. 

For example, we must rule out the hypothetical situation in which $A_x \neq A_y$ can be achieved for $N_{S,R}$, but only with $\dim(A_x \cap A_y)=c-1$. Otherwise, being unable to bound $c$, we will not be able to pick a fixed $\varepsilon>0$ such that $W$ satisfying $\frac{1}{c}\dim(W)=\varepsilon$ can be found. A key tool here is \cref{lem:rank_reg_rep}.

We actually need a little more: it is necessary to obtain similar bounds for the overlap of $A_x$ with any subspace in $\U$. That this can be done is a consequence of the following definition and proposition.

\begin{defn}
	Let $M=\left(E,r\right)$ be a triangle matroid with distinguished
	basis $B$ and let $\mathcal{A}=\left\{ A_{e}\right\} _{e\in E}$
	be a weak $c$-representation of $M$.
	We call an extension $\mathcal{U}=\left\{ U_{e}\right\} _{e\in E}$
	of $\mathcal{A}$ \emph{well-separated }with respect to a given
	$x\in E$ if for any $T\subseteq E$, either $A_{x}\subseteq U_{T}$ or
	$\dim\left(A_{x}\cap U_{T}\right)\le\frac{1}{2}c$.
\end{defn}
The next proposition shows that the full extensions of certain weak $c$-representations of the matroids $N_{S,R}$ are well-separated.

\begin{prop}\label{prop:well_separated}
	Let $\left\langle S\mid R\right\rangle $ be a finite
	presentation and let $G$ be a finite group together with a homomorphism $\varphi:G_{S,R}\rightarrow G$.
	Set $n\coloneqq |G|$ and let $\mathcal{A}= \mathcal{A}_{G,\varphi}$ be the weak $n$-representation
	of the matroid $N_{S,R}=(E_{S,R},r)$ with respect to the distinguished basis $B=\left\{ b^{\left(1\right)},b^{\left(2\right)},b^{\left(3\right)}\right\} $ constructed in \cref{pro:G_weak_representation}.

	Let $\mathcal{U}=\left\{ U_{e}\right\} _{e\in E}$
	be an extension of $\A$ and assume that $U_{T}$ is full in $\mathcal{U}$ for any $T\subseteq E_{S,R}$.
	Then $\mathcal{U}$ is well-separated with respect
	to $x^{\left(1\right)}\in E_{S,R}$ for any $x\in S$.
\end{prop}
\begin{proof}
	Fix some $x\in S$ and $T\subseteq E_{S,R}$.
	We have to show $U_{T}\cap A_{x^{\left(1\right)}}$
	is either equal to $A_{x^{\left(1\right)}}$ or has dimension at most
	$\frac{1}{2}n$.
	We split this into cases based on the value of $r\left(T\right)$,
	noting that $U_{T}\cap A_{x^{\left(1\right)}}=U_{T}\cap U_{B}\cap A_{x^{\left(1\right)}}$,
	since $A_{x^{\left(1\right)}}\subseteq U_{B}=A_{B}$.
	\begin{description}
		\item[Case 1: $r\left(T\right)=1$] In this case $T=\left\{ t\right\} $ for some
		$t\in E_{S,R}$, and $U_{T}\cap U_{B}=A_{t}$. If $t$ is not a bottom
		element of $N_{S,R}$, then $\left\{ t,x^{\left(1\right)}\right\} $
		is full in $\mathcal{A}$ by \Cref{rmrk:triangle_reg}.
		It has rank $2$ in $N_{S,R}$, so 
		$
		\dim\left(A_{t}+A_{x^{\left(1\right)}}\right)=2n,
		$
		implying $A_{t}\cap A_{x^{\left(1\right)}}=0$. If $t$ is a bottom
		element of $N_{S,R}$, then by \cref{pro:G_weak_representation},
		\[
		A_{x^{\left(1\right)}}=\text{colspan}\begin{bsmallmatrix}-I_{n}\\
		\rho\left(\varphi\left(x\right)\right)\\
		0
		\end{bsmallmatrix},\quad A_{t}=\text{colspan}\begin{bsmallmatrix}-I_{n}\\
		P_{t}\\
		0
		\end{bsmallmatrix}
		\]
		respectively, where $\rho$ is the regular representation of $G$
		and $P_{t}$ is some matrix either of the form $\rho\left(\varphi\left(g\right)\right)$
		for some $g\in G$, or $0$ if $t\in B$.
		Thus by \cref{lem:block_matrices} (a),
		\[
		\dim\left(A_{x^{(1)}}+A_{t}\right)=n+\rk\left(\rho\left(\varphi\left(x\right)\right)-P_{t}\right)\ge\frac{3}{2}n
		\]
		in either case: if $P_{t}=0$ this is trivial, and otherwise
		it follows from \cref{lem:rank_reg_rep} since $n=|G|$.
		This implies $A_{x^{\left(1\right)}}\cap A_{t}$
		has dimension at most $\frac{1}{2}n$ as required.
		\item[Case 2: $r\left(T\right)=2$]
		If $x^{(1)} \in T$ then $A_{x^{\left(1\right)}} \subseteq U_T\cap U_B$ holds trivially.
		If $T$ is contained in a side $\ell$ of the triangle matroid $N_{S,R}$ then $U_T\cap U_B$ either contains $A_{x^{\left(1\right)}}$ if $\ell$ is the bottom side or it intersects $A_{x^{\left(1\right)}}$ trivially if $\ell$ is the left or right side of $N_{S,R}$.
		
		If $T$ is not contained in such a line, note that $U_{T}\cap U_{B}=U_{T}\cap A_{B}\supseteq A_{T}\cap A_{B}$.
		By \Cref{rmrk:triangle_reg}, $T$ is full in $\mathcal{A}$.
		Hence, $\dim\left(A_{T}\cap A_{B}\right)=\dim\left(A_{T}\right)=2n$.
		Since $\dim\left(U_{T}\cap U_{B}\right)=2n$, the containment implies equality: $U_{T}\cap U_{B}=A_{T}\cap A_{B}=A_{T}$.
		
		Now suppose $y^{(1)}\in T$ for some $y\in S$ with $x\neq y$.
		We allow $y^{\left(1\right)}$ to be $b^{\left(k\right)}$ for $k=1,2$ or $e^{\left(1\right)}$.
		Since $A_T$ intersects $A_{b^{(1)},b^{(2)}}$ in $A_{y^{(1)}}$ we have $A_{T}\cap A_{x^{\left(1\right)}}=A_{y^{\left(1\right)}}\cap A_{x^{\left(1\right)}}$
		Thus, we may reduce to the case $r\left(T\right)=1$ by taking $T'\coloneqq \left\{ y^{\left(1\right)}\right\} $.
		
		Lastly, assume $\left\{ y^{\left(2\right)},z^{\left(3\right)}\right\} \subseteq T$ for some $y,z\in S\cup\{e\}$.
		The intersection $A_{T}\cap A_{\left\{ b^{\left(1\right)},b^{\left(2\right)}\right\}} $ has dimension at most $n$, since the sum
		$
		A_{T}+A_{\left\{ b^{\left(1\right)},b^{\left(2\right)}\right\} }
		$
		is the entire space $A_{B}$, and has dimension $3n$ (where $\dim A_{T}=2n$).
		
		The block columns representing
		$y^{\left(2\right)},y^{\left(3\right)}$ in $\mathcal{A}$ are
		$
		\begin{bsmallmatrix}0\\
		-I_{n}\\
		\rho\left(g_{2}\right)
		\end{bsmallmatrix},\begin{bsmallmatrix}\rho\left(g_{3}\right)\\
		0\\
		-I_{n}
		\end{bsmallmatrix}
		$
		for some $g_{2},g_{3}\in G$ respectively where $\rho$ is the regular representation of $G$. 
		By \Cref{lem:block_matrices} b)	we see that the matrix
		\[
		\begin{bsmallmatrix}-I_{n} & 0 & \rho\left(g_{3}\right)\\
		\rho\left(g_{2}\right)^{-1}\rho\left(g_{3}\right)^{-1} & -I_{n} & 0\\
		0 & \rho\left(g_{2}\right) & -I_{n}
		\end{bsmallmatrix}
		\]
		has rank $2n$, where the first block column is in $A_{\left\{ b^{\left(1\right)},b^{\left(2\right)}\right\} }$.
		Thus, the intersection above is given by the block column span of the first block
		column.
		By the proof of the case $r\left(T\right)=1$ and the fact
		that $\rho\left(g_{2}\right)^{-1}\rho\left(g_{3}\right)^{-1}=\rho(\left(g_{3}g_{2}\right)^{-1}),			$
		we see
		\[
		A_{x^{\left(1\right)}}\cap\text{colspan}\begin{bsmallmatrix}-I_{n}\\
		\rho\left(g_{2}\right)^{-1}\rho\left(g_{3}\right)^{-1}\\
		0
		\end{bsmallmatrix}\le\frac{1}{2}n,
		\]
		since we can assume $\phi(x)\neq \left(g_{2}g_3\right)^{-1}$ (otherwise we would be in the previous case). This implies the claim.
		\item[Case 3: $r\left(T\right)=3$] In this case $U_{T}\cap U_{B}$ has dimension $3n$ since $T$ is full in $\mathcal{U}$.
		Using the fact $\dim\left(U_{B}\right)=3n$ we obtain $U_{T}\supseteq U_{B}\supseteq A_{x^{\left(1\right)}}$ as desired.\qedhere
	\end{description}
\end{proof}

\subsection{Forcing an Inequality}\label{sec:inequality}

\begin{defn}
	The \emph{double} of a subspace arrangement $\mathcal{U}=\left\{ U_{e}\right\} _{e\in E}$
	in a vector space $V$ is the arrangement $\mathcal{W}=\left\{ W_{e}\right\} _{e\in E}$
	in $V\oplus V$, where
	\[
	W_{e}\coloneqq U_{e}\oplus U_{e}.
	\]
	Note that in this situation, $r_{\mathcal{W}}^{2c}=r_{\mathcal{U}}^{c}$ and $r_{\mathcal{W}}^{c}=2r_{\mathcal{U}}^{c}$.
	Hence if $r_{\U}^c$ represents some polymatroid so does $r^{2c}_{\mathcal{W}}$.  

	When $\U$ represents some polymatroid $c\cdot g$, or extends a weak $c$-representation $\A$ of some matroid $M$, it will be convenient to think of $\mathcal{W}$ not just as a double of $\U$, but also as an arrangement representing $2c\cdot g$ or an extension of the weak $2c$-representation $\A\oplus\A$ of $M$.
	
	The point is that non-integer multiples of $c$ are not accessible via the combinatorial rank function, but to use the tools of the previous section we need to work with subspaces of dimension $\frac{c}{2}$. Doubling a $c$-arrangement, and considering it as a $c'=2c$-representation of the same polymatroid, makes $\frac{c'}{2}=c$ the ``basic unit of measurement'' (while also guaranteeing that it is an integer, i.e. that $c'$ is even).
\end{defn}

\textbf{Notation.} We use the following notation in the rest of this section:

Let $M=\left(E,r\right)$ be a triangle matroid
	with respect to the distinguished basis $B=\left\{ b^{\left(1\right)},b^{\left(2\right)},b^{\left(3\right)}\right\} $,
	and let $\mathcal{W}=\left\{ W_{e}\right\} _{e\in E}$ be a $2c$-admissible extension
	of the weak $2c$-representation $\mathcal{A}=\left\{ A_{e}\right\} _{e\in E}$
	of $M$. Suppose $x\in E$ is in the line $\ell$ of $M$ spanned by $\left\{ b^{\left(1\right)},b^{\left(2\right)}\right\} $, i.e. the bottom line of $M$.
\begin{defn}\label{defn:sep_basis}	
	Let 
	\[
	\mathcal{W}^{\mathcal{B}} \coloneqq \left\{ W_{e,i}\right\} _{\substack{e\in E,\\
			1\le i\le d_{e}
		}
	}
	\]
	be a $c$-basis for $\mathcal{W}$, where $d_{e}\coloneqq \frac{1}{c}\dim W_{e}$
	for each $e\in E$. 
	
	Let $y\in E$ be an element on the bottom line of $M$.
	Suppose $\mathcal{W}^{\mathcal{B}}$ has the following property: $W_{x,1}\subseteq W_{b^{\left(1\right)}}+W_{b^{\left(2\right)}}$,
	and $W_{x,1}\cap W_{y}=0$. In this situation, we say $\mathcal{W}^{\mathcal{B}}$
	\emph{separates $x$ from $y$.}
\end{defn}

Note that by construction each $d_e$  is an even number and and the subspaces $W_{e,i}$ are generic $c$-dimensional subspaces of $W_e$ for each $e\in E$.

The following proposition implies that separating a pair of elements is a combinatorial property, depending only on the rank function. It also describes the main consequence: in the weak $c$-arrangement $\A$ corresponding to $\mathcal{W}$, the subspaces corresponding to $x$ and to $y$ are distinct.

\begin{prop}\label{prop:comb_separation}
\

	\begin{enumerate}
		\item If the arrangement $\mathcal{W}$ has a $c$-basis $\mathcal{W}^{\mathcal{B}}$
		which separates $x$ from $y$ then $A_{x}\neq A_{y}$. 
		\item A $c$-basis $\mathcal{W}^{\mathcal{B}}$ separates $x$ from $y$
		if and only if the following two conditions hold:
		\[
		r_{\mathcal{W}^{\mathcal{B}}}^{c}\left(\left\{ \left(x,1\right)\right\} \cup\left\{ \left(b^{\left(i\right)},j\right)\right\} _{1\le i,j\le2}\right)=4,
		\]
		\[
		r_{\mathcal{W}^{\mathcal{B}}}^{c}\left(\left\{ \left(x,1\right)\right\} \cup\left\{ \left(y,i\right)\right\} _{1\le i\le d_{y}}\right)=d_{y}+1.
		\]
	\end{enumerate}
\end{prop}
\begin{proof}
\	
	\begin{enumerate}
		\item Suppose the $c$-basis $\mathcal{W}^{\mathcal{B}}$ of $\mathcal{W}$
		separates $x$ and $y$. Then 
		\[
		W_{x,1}\subseteq W_{x}\cap W_{B}=A_x
		\]
		is not contained in $W_{y}\supseteq A_{y}$. Therefore
		$A_{x}$ is also not contained in $A_{y}$,
		and $A_{x}\neq A_{y}$.
		\item The condition $W_{x,1}\subseteq W_{b^{\left(1\right)}}+W_{b^{\left(2\right)}}$
		is equivalent to the equation 
		\[
		W_{b^{\left(1\right)}}+W_{b^{\left(2\right)}}=W_{x,1}+W_{b^{\left(1\right)}}+W_{b^{\left(2\right)}},
		\]
		and this occurs if and only if both sides have the same dimension,
		which is $4c$ by construction.
		
		The condition $W_{x,1}\cap W_{y}=0$ is equivalent to the equation
		$\dim\left(W_{x,1}+W_{y}\right)=\dim\left(W_{x,1}\right)+\dim\left(W_{y}\right)$,
		and by construction the right-hand side equals $cd_{y}+c$.\qedhere
	\end{enumerate}
\end{proof}
\textbf{Construction.} Suppose $\mathcal{W}$ is well-separated with respect to some $x\in E$. We
construct the following $c$-basis, which as we will show separates
$x$ from any $y\in\ell$ such that $A_{x}\neq A_{y}$:
\begin{enumerate}
	\item For each $z\in E\backslash\left\{ x\right\} $, let $\left\{ W_{z,1},\ldots,W_{z,d_{z}}\right\} $
	be generic $c$-dimensional subspaces of $W_{z}$.
	\item Let $W_{x,1}$ be a generic $c$-dimensional subspace of $W_{x}\cap W_{\left\{ b^{\left(1\right)},b^{\left(2\right)}\right\} }=A_x.$
	\item Let $\left\{ W_{x,2},\ldots,W_{x,d_{x}}\right\} $ be generic $c$-dimensional
	subspaces of $W_{x}$.
\end{enumerate}
The collection of all these subspaces will be denoted $\mathcal{W}^{\mathcal{B},x}$.
\begin{lemm}\label{lemm:exists_sep_basis1}
	$\mathcal{W}^{\mathcal{B},x}$ is a $c$-basis of $\mathcal{W}$.
\end{lemm}
\begin{proof}
	We need to show $\mathcal{W}^{\mathcal{B},x}$ is a $c$-arrangement
	and that $\sum_{i=1}^{d_{z}}W_{e,i}=W_{e}$ for each $e\in E$. The second
	part is clear from genericity. Let us prove $\mathcal{W}^{\mathcal{B},x}$
	is a $c$-arrangement. 
	
	By construction, the subspaces $\mathcal{W}_{z,i}^{\mathcal{B},x}$
	(for any $\left(z,i\right)\neq\left(x,1\right)$) are generic subspaces
	of the subspaces $W_{z}$, each of dimension $c$. Also, $\left\{ W_{z}\right\} _{z\in E}\cup\left\{ W_{x,1}\right\} $
	is a $c$-admissible subspace arrangement, essentially since $\mathcal{W}$ is well-separated:
	If $T\subseteq E$ is any subset, then $W_{T}\cap A_{x}$
	is either of dimension at most $c$ or $A_{x}\subseteq W_{T}$.
	In the first case, $W_{T}$ intersects $A_{x}$ in dimension
	at most $c$, and since $W_{x,1}$ is generic of dimension $c$ in
	$A_{x}$ it satisfies $W_{x,1}\cap W_{T}=0$. In the second
	case, $A_{x}\subseteq W_{T}$ so also $W_{x,1}\subseteq W_{T}$,
	and $W_{x,1}+W_{T}=W_{T}$ again has dimension a multiple of $c$.
	Hence $\mathcal{W}\cup\left\{ W_{x,1}\right\} $ is $c$-admissible
	and the claim follows from the Splitting Lemma \ref{lemm:splitting}: $\mathcal{W}^{\mathcal{B},x}$
	is a subspace arrangement comprised of some generic $c$-dimensional
	subspaces of the $\left\{ W_{z}\right\} _{z\in E}$ and of $W_{x,1}$
	(note that any $c$-dimensional subspace of $W_{x,1}$ is equal to $W_{x,1}$).
\end{proof}

\begin{prop}\label{prop:sep_basis}
	In the notation of the construction, $\mathcal{W}^{\mathcal{B},x}$
	separates $x$ from $y$ for any $y\in\ell$ such that $A_{x}\neq A_{y}$.
\end{prop}
\begin{proof}
	The equality
	\[
	r_{\mathcal{W}^{\mathcal{B},x}}^{c}\left(\left\{ \left(x,1\right)\right\} \cup\left\{ \left(b^{\left(i\right)},j\right)\right\} _{1\le i,j\le2}\right)=4
	\]
	is clear from the construction as $W_{x,1}$ is contained in the $4c$-dimensional subspace
	\[
	W_{b^{\left(1\right)}}+W_{b^{\left(2\right)}}=\sum_{i,j=1}^{2}W_{b^{\left(i\right)},j}.
	\]
	Suppose $A_{x}\neq A_{y}$. Then $A_{x}\not\subseteq W_{y}$, as otherwise
	\[
	A_{y}=W_{y}\cap W_{B}=\left(W_{y}+A_{x}\right)\cap W_{B}\supseteq A_{x}
	\]
	gives a contradiction. Hence $A_{x}\cap W_{y}$ has dimension at most
	$c$ by well-separatedness.
	As $W_{x,1}\subseteq A_{x}$ is generic of dimension $c$, it
	has trivial intersection with $W_{y}$. This proves that
	\begin{align*}
	r_{\mathcal{W}^{\mathcal{B},x}}^{c}\left(\left\{ \left(x,1\right)\right\} \cup\left\{ \left(y,i\right)\right\} _{1\le i\le d_{y}}\right) &=\frac{1}{c}\dim\left(W_{x,1}+W_{y}\right) \\
	&=\frac{1}{c}\left(\dim\left(W_{x,1}\right)+\dim W_{y}-\dim\left(W_{x,1}\cap W_{y}\right)\right)
	\end{align*}
	equals $\frac{1}{c}\left(\dim\left(W_{x,1}\right)+\dim\left(W_{y}\right)\right)=1+d_{y}$.
	Therefore by \Cref{prop:comb_separation}, the $c$-basis~$\mathcal{W}^{\mathcal{B},x}$
	separates $x$ from $y$.
\end{proof}

To later use the results of this section independently of the specific notation introduced above, we encode the second part of \cref{prop:comb_separation} in a definition and its first part in a lemma.
\begin{defn}
	Let $M=(E,r)$ be a triangle matroid with distinguished basis $B=\{b^{(1)},b^{(2)},b^{(3)}\}$, and let $g$ be a polymatroid extension of $M$.
	Let $x\in E$ be a bottom element of $M$, i.e. one which lies in the flat spanned by $\{b^{(1)},b^{(2)}\}$. We will say a polymatroid expansion (in the sense of~\Cref{def:expansion} c))
	$\left(\{(e,i)\}_{\substack{e\in E, \\ 1\le i\le 2g(e)}} , r_{\text{exp}} \right)$
	of the doubled polymatroid $2g$ separates $x$ from $y$ if
	\begin{align*}
	r_\text{exp}\left(\left\{ \left(x,1\right)\right\} \cup\left\{ \left(b^{\left(i\right)},j\right)\right\} _{1\le i,j\le2}\right)&=4 \mbox{ and }\\
	r_\text{exp}\left(\left\{ \left(x,1\right)\right\} \cup\left\{ \left(y,i\right)\right\} _{1\le i\le 2g(y)}\right)&=2g(y)+1.
	\end{align*}
\end{defn}

\begin{lemm}
	Let $M=(E,r)$ be a triangle matroid with distinguished basis $B=\{b^{(1)},b^{(2)},b^{(3)}\}$, and let $g$ be a polymatroid extending $M$. Suppose an expansion $N$ of $2g$ separates $x$ from $y$. Given a $c$-arrangement representing $N$, let $\mathcal{W}$ be the corresponding subspace arrangement representing $2c\cdot g$, and let $\A=\{A_e\}_{e\in E}$ be the corresponding weak $2c$-representation of $M$ (as constructed in \cref{thm:intersection_thm}).  Then $A_x \neq A_y$.
\end{lemm}

This is a direct consequence of \cref{prop:comb_separation}.

The next proposition is used in the theorem which follows. The theorem puts together several ideas from this and preceding sections.

\begin{prop}\label{prop:sec_8}
	Let $M=(E,r)=N_{S,R}$ be a triangle matroid constructed from a group presentation $\langle S \mid R\rangle$ as in \cref{def:matroid_N}. Denote its distinguished basis by $B$. Let $x,y\in E \setminus B$ be elements on the bottom line of $M$ Let $g$ be the polymatroid extending $M$ constructed in \cref{thm:regularization}.
	
	Suppose there exists a weak $c$-representation $\A = \{A_e\}_{e\in E}$ of $M$ such that $A_x \neq A_y$. Then there exists a $c$-arrangement representation of some expansion of $2g$ which separates $x$ from $y$.
\end{prop}
\begin{proof}
	By \cref{pro:weak_representation_G}, there is a finitely generated matrix group $G_{\A}$ and a homomorphism $\phi:G_{S,R}\rightarrow G_{\A}$ corresponding to the weak $c$-representation $\A$. By construction, the elements $x,y \in E \setminus B$ correspond to two elements of $G_{S,R},$ and since $A_x \neq A_y$ their images in $G_{\A}$ are distinct.
	Using Malcev's theorem, $G_{\A}$ has a finite quotient $G'_{\A}$ in which the images of the elements corresponding to $x,y$ are distinct.
	Replace $G_{\A}$ by $G'_{\A}$, and the homomorphism $\phi$ by its composition with the quotient map. 
	
	Define $\A' = \A_{G_{\A},\phi} = \{A'_e\}_{e\in E}$ as in \cref{pro:G_weak_representation}, and note that also $A'_x \neq A'_y$. Let $\U$ be the extension of $\A'$ constructed in \cref{thm:regularization}. The arrangement $\U$ is well-separated with respect to $x$ by \cref{prop:well_separated}.
	
	By the construction of this section and its properties shown in \cref{lemm:exists_sep_basis1} and \cref{prop:sep_basis}, the double $\mathcal{W}$ of $\U$ has a separating $c$-basis; this basis is, by definition, a $c$-arrangement representing an expansion of $2g$ which separates $x$ and $y$.
\end{proof}

\begin{theorem}\label{thm:sec_8}
	Let $M=N_{S,R}=(E,r)$ be a triangle matroid constructed from a finite presentation $\langle S \mid R\rangle$ as in \cref{def:matroid_N}. Denote its distinguished basis by $B$, and let $x,y\in E$ be elements on the bottom line of $M$. Then there exists a finite set of matroids $\{M_i\}_{i=1}^n$ such that $M$ has a weak $c$-representation $\A = \{A_e\}_{e\in E}$ with $A_x \neq A_y$ for some $c$ if and only if there is at least one $1\le i\le n$ such that $M_i$ is representable as a $c'$-arrangement for some $c'\ge 1$.
\end{theorem}
\begin{proof}
	Let $g$ be the polymatroid constructed from $M$ in \cref{thm:regularization}. Among the expansions of $2g$ take $\{M_i\}_{i=1}^n$ to be all those which separate $x$ from $y$. There are finitely many of these, since $2g$ has finitely many expansions (see the remark following \cref{lemm:c_basis_expansion2}). By \cref{lemm:c_basis_expansion1} a representation of any $M_i$ as a $c$-arrangement gives a representation of the polymatroid $2c\cdot g$. This yields a weak $2c$-representation $\A=\{A_e\}_{e\in E}$ of $M$ by the intersection procedure of \cref{thm:intersection_thm}. Since $M_i$ is an expansion of $2g$ separating $x,y$, $A_x \neq A_y$.
	
	Conversely, if $M$ 	has a weak $c$-representation $\A = \{A_e\}_{e\in E}$ with $A_x \neq A_y$, then \cref{prop:sec_8}
	implies that there is a $c$-arrangement representing an expansion of $2g$ which separates $x,y$, and this is a $c$-arrangement representation of $M_i$ for some $1\le i\le n$.
\end{proof}

\section{Proof of \Cref{thm:undecidable_matroids}}\label{sec:main_proof}

In this section we connect our previous results to prove \Cref{thm:undecidable_matroids}.

\begin{theorem}
	For each instance of the uniform word problem for finite groups, there exists a finite set of matroids $\{M_1,\ldots,M_n\}$ (computable from the given instance of the problem,) such that at least one of them is representable as a $c$-arrangement if and only if the answer to the given instance of the UWPFG is negative.
\end{theorem}
\begin{proof}
	Let $\langle S \mid R\rangle$ be a finite presentation of a group and $w\in S$. Let $N_{S,R}$ be the corresponding matroid on the ground set $E_{S,R}$ with distinguished basis $B = \{b^{(1)},b^{(2)},b^{(3)}\}$, as constructed in \cref{def:matroid_N}. Let $M_1, \ldots, M_n$ be the matroids on $E_{S,R}$ constructed from $N_{S,R}$ and $w^{(1)}, e^{(1)} \in E_{S,R}$ in \cref{thm:sec_8}.
	
	By \cref{thm:sec_8}, at least one of the matroids $M_1,\ldots,M_n$ is representable as a $c$-arrangement if and only if there exists a weak $c$-representation $\A = \{A_e\}_{e\in E_{S,R}}$ of $N_{S,R}$ such that $A_{w^{(1)}} \neq A_{e^{(1)}}$. This occurs if and only if the solution to the UWPFG instance is negative by \cref{theo:weak_undecidable}.
\end{proof}

Hence by Slobodskoi's theorem~\cite{Slo81}, existence of $c$-arrangement representations of matroids is undecidable.

\bibliographystyle{myalpha}
\bibliography{matroid.bib}

\providecommand{\bysame}{\leavevmode\hbox to3em{\hrulefill}\thinspace}
\providecommand{\MR}{\relax\ifhmode\unskip\space\fi MR }
\providecommand{\MRhref}[2]{%
  \href{http://www.ams.org/mathscinet-getitem?mr=#1}{#2}
}
\providecommand{\href}[2]{#2}
\begin{thebibliography}{BBEPT14}

\bibitem[ANLY00]{Ahl00}
R.~{Ahlswede}, {Ning Cai}, S.~.~R. {Li}, and R.~W. {Yeung}, \emph{Network
  information flow}, IEEE Transactions on Information Theory \textbf{46}
  (2000), no.~4, 1204--1216.

\bibitem[BBEPT14]{BBP14}
Amos Beimel, Aner Ben-Efraim, Carles Padr\'{o}, and Ilya Tyomkin,
  \emph{Multi-linear secret-sharing schemes}, Theory of cryptography, Lecture
  Notes in Comput. Sci., vol. 8349, Springer, Heidelberg, 2014, pp.~394--418.

\bibitem[BD91]{BD91}
Ernest~F. Brickell and Daniel~M. Davenport, \emph{On the classification of
  ideal secret sharing schemes}, Journal of Cryptology \textbf{4} (1991),
  no.~2, 123--134.

\bibitem[Bj{\"o}94]{Bjo92}
Anders Bj{\"o}rner, \emph{Subspace arrangements}, First {E}uropean {C}ongress
  of {M}athematics, {V}ol.\ {I} ({P}aris, 1992), Progr. Math., vol. 119,
  Birkh\"auser, Basel, 1994, pp.~321--370.

\bibitem[DFZ07]{DFZ07}
R.~{Dougherty}, C.~{Freiling}, and K.~{Zeger}, \emph{Networks, matroids, and
  non-shannon information inequalities}, IEEE Transactions on Information
  Theory \textbf{53} (2007), no.~6, 1949--1969.

\bibitem[Dow73]{Dow73}
Thomas~A. Dowling, \emph{A class of geometric lattices based on finite groups},
  Journal of Combinatorial Theory, Series B \textbf{14} (1973), no.~1, 61 --
  86.

\bibitem[ESG10]{ElR10}
S.~{El Rouayheb}, A.~{Sprintson}, and C.~{Georghiades}, \emph{On the index
  coding problem and its relation to network coding and matroid theory}, IEEE
  Transactions on Information Theory \textbf{56} (2010), no.~7, 3187--3195.

\bibitem[Fuj78]{Fuj78}
Satoru Fujishige, \emph{Polymatroidal dependence structure of a set of random
  variables}, Information and Control \textbf{39} (1978), 55--72.

\bibitem[GM88]{GM88}
Mark Goresky and Robert MacPherson, \emph{Stratified {M}orse theory},
  Ergebnisse der Mathematik und ihrer Grenzgebiete (3), vol.~14,
  Springer-Verlag, Berlin, 1988.

\bibitem[KPY]{KPY20}
Lukas Kühne, Rudi Pendavingh, and Geva Yashfe, \emph{On the comparison of
  skew-linear and multilinear matroids}, in preparation.

\bibitem[LS77]{LS77}
Roger~C. Lyndon and Paul~E. Schupp, \emph{Combinatorial group theory},
  Springer-Verlag, Berlin-New York, 1977, Ergebnisse der Mathematik und ihrer
  Grenzgebiete, Band 89.

\bibitem[Mal40]{Mal40}
Anatoly Malcev, \emph{On isomorphic matrix representations of infinite groups},
  Rec. Math. [Mat. Sbornik] N.S. \textbf{8 (50)} (1940), 405--422.

\bibitem[{Mat}99]{Mat93}
Franti\v{s}ek {Mat\'{u}\v{s}}, \emph{Matroid representations by partitions},
  Discrete Math. \textbf{203} (1999), no.~1-3, 169--194.

\bibitem[Mn{\"e}88]{Mne88}
Nikolai~E. Mn{\"e}v, \emph{The universality theorems on the classification
  problem of configuration varieties and convex polytopes varieties}, Topology
  and geometry---{R}ohlin {S}eminar, Lecture Notes in Math., vol. 1346,
  Springer, Berlin, 1988, pp.~527--543.

\bibitem[Ngu86]{Ngu86}
Hien~Q. Nguyen, \emph{Semimodular functions}, Encyclopedia of Mathematics and
  its Applications, p.~272–297, Cambridge University Press, 1986.

\bibitem[Oxl11]{Oxl11}
James Oxley, \emph{Matroid theory}, second ed., Oxford Graduate Texts in
  Mathematics, vol.~21, Oxford University Press, Oxford, 2011.

\bibitem[PvZ13]{PvZ13}
Rudi~A. Pendavingh and Stefan~H.M. van Zwam, \emph{Skew partial fields,
  multilinear representations of matroids, and a matrix tree theorem}, Advances
  in Applied Mathematics \textbf{50} (2013), no.~1, 201 -- 227.

\bibitem[Rad57]{Rad57}
Richard Rado, \emph{Note on independence functions}, Proc. London Math. Soc.
  (3) \textbf{7} (1957), 300--320.

\bibitem[SA98]{SA98}
Juriaan Simonis and Alexei Ashikhmin, \emph{Almost affine codes}, Designs,
  Codes and Cryptography \textbf{14} (1998), no.~2, 179--197.

\bibitem[Slo81]{Slo81}
A.~M. Slobodskoi, \emph{Unsolvability of the universal theory of finite
  groups}, Algebra and Logic \textbf{20} (1981), no.~2, 139--156.

\bibitem[Stu87]{Stu87}
Bernd Sturmfels, \emph{On the decidability of {D}iophantine problems in
  combinatorial geometry}, Bull. Amer. Math. Soc. (N.S.) \textbf{17} (1987),
  no.~1, 121--124.

\bibitem[Zas94]{Zas94}
Thomas Zaslavsky, \emph{Frame matroids and biased graphs}, European Journal of
  Combinatorics \textbf{15} (1994), no.~3, 303 -- 307.

\end{thebibliography}

\end{document}